\documentclass[12pt,reqno]{amsart}
\usepackage{amsmath, amsthm, amssymb, stmaryrd}
\usepackage{hyperref}
\usepackage{enumerate}
\usepackage{url}

\usepackage{mathtools}

\usepackage{multirow, array}
\usepackage{placeins}

\usepackage{caption} 
\advance \topmargin by -\headheight
\advance \topmargin by -\headsep
     
\evensidemargin \oddsidemargin
\newtheorem{thm}{Theorem}[section]
\newtheorem{lemma}[thm]{Lemma}

\newtheorem{cor}[thm]{Corollary}

\theoremstyle{definition}

\newtheorem{example}[thm]{Example}

\newtheorem{principle}[thm]{Principle}

\theoremstyle{remark}
\newtheorem{remark}[thm]{Remark}

\numberwithin{equation}{section}

\newcommand*\wrapletters[1]{\wr@pletters#1\@nil}
\def\wr@pletters#1#2\@nil{#1\allowbreak\if&#2&\else\wr@pletters#2\@nil\fi}

\def\eps{\varepsilon}

\def\le{\leqslant} \def\ge{\geqslant}

\def\d{{\,{\bRm d}}}

\def \bN {\mathbb N}

\def \bQ {\mathbb Q}
\def \bR {\mathbb R}
\def \bZ {\mathbb Z}

\def \ba {\mathbf a}
\def \bb {\mathbf b}

\def \bh {\mathbf h}

\def \bk {\mathbf k}

\def \bp {\mathbf p}
\def \bq {\mathbf q}

\def \bu {\mathbf u}

\def \bx {\mathbf x}

\def \by {\mathbf y}

\def \fm {\mathfrak m}

\def \fA {\mathfrak A}

\def \fC {\mathfrak C}
\def \fD {\mathfrak D}

\def \fM {\mathfrak M}

\def \fS {\mathfrak S}

\def \det {\mathrm{det}}

\begin{document}
\title[Random Diophantine equations in the primes]{Random Diophantine equations in the primes}
\author[Philippa Holdridge]{Philippa Holdridge}
\address{}
\email{Philippa.holdridge@warwick.ac.uk}
\subjclass[2020]{11P55,11D72,11P32,11E76}
\keywords{hasse principle, waring-goldbach, circle method, geometry of numbers}
\thanks{}
\date{}

\begin{abstract}We consider equations of the form $a_{1}x_{1}^{k}+\cdots+a_{s}x_{s}^{k}=0$ where the variables $x_{i}$ are all taken to be primes. We define an analogue of the Hasse principle for solubility in the primes (which we call the prime Hasse principle), and prove that, whenever $k\ge2$, $s\ge 3k+2$, this holds for almost all such equations. This is based on work of Brüdern and Dietmann on the Hasse principle \cite{bd2014}. We then prove some further results about prime solubility and the prime Hasse principle, including a partial converse, and some counterexamples. Of particular interest are counterexamples of degree 2, which show that the analogue of the Hasse-Minkowski theorem fails for prime solubility.\end{abstract}
\maketitle
\setcounter{tocdepth}{1}
\tableofcontents

\section{Introduction}
\label{intro}
If a Diophantine equation is solvable in $\bQ$, then it is also solvable in $\bR$ and in $\bQ_{p}$ for all primes $p$. If the converse also holds, then we say that the equation satisfies the Hasse principle. In this paper, we discuss equations of the form:
\begin{equation}\label{maineq}
a_{1}x_{1}^{k}+\cdots+a_{s}x_{s}^{k}=0,
\end{equation}
where $a_{1},\dots,a_{s}$ are nonzero integer constants and $x_{j}$ are prime variables. We write $\ba=(a_{1},\dots,a_{s})$ to mean a tuple of integers, which are always assumed to be nonzero unless stated otherwise, and we let $|\ba|=\max |a_{j}|$.

 The circle method is a useful tool for proving that the Hasse principle holds for certain families of equations, provided there are sufficiently many variables. For example, it can be shown that when $k$ is odd or $a_{j}$ do not all have the same sign, \eqref{maineq} has nontrivial integer solutions for all sufficiently large $s$, so that in particular the Hasse principle holds \cite[Theorem 9.1]{vaughan}. Brüdern and Dietmann showed \cite{bd2014} that the Hasse principle still holds for $s\ge 3k+2$ ``almost always''. More precisely, if we randomly choose an equation in the box $|\ba|\le A$ then the Hasse principle holds with probability $1-o(1)$ as $A\rightarrow \infty$. In the first part of this paper, we will show an analogous result, but where we are interested only in solutions in the primes. For this we will need an analogue of the Hasse principle which we call the prime Hasse principle. Let $\bR^{+}=\{x\in \bR: x>0\}$ and $\bZ_{p}^{\times}=\{x \in \bQ_{p}:|x|_{p}=1\}$.
\begin{principle}\label{pHp}[Prime Hasse principle]
Let $f_{1},\dots,f_{r}\in \bZ[x_{1},\dots,x_{s}]$ be homogeneous. We say that the system of equations $f_{1}(\bx)=\cdots=f_{r}(\bx)=0$ satisfies the prime Hasse principle if the following holds:

If $f_{1}(x_{1},\dots,x_{s})=\cdots=f_{r}(x_{1},\dots,x_{s})=0$ has a solution with $x_{j}\in \bR^{+}$ and for each $p$ it has a solution with $x_{j}\in \bZ_{p}^{\times}$, then it has a solution in the set of prime numbers.
\end{principle}
While we have stated the principle for general homogeneous systems, we will only concern ourselves with single equations which are of the form \eqref{maineq}, with $k\ge2$. The reason for this choice of definition is that the condition of solubility in $\bR^{+}$ and $\bZ_{p}^{\times}$ will turn out to be equivalent to the main term in the circle method being nonzero. Before we proceed, we note some differences between the Hasse principle and prime Hasse principle. Firstly, while the converse to the Hasse principle always holds, this is not true for the prime Hasse principle. For example, take $p^{2}x^{2}-q^{2}y^{2}=0$, where the constants $p$ and $q$ are distinct primes. This has a solution $x=q$, $y=p$, but it does not have a solution in $\bZ_{p}^{\times}$ or $\bZ_{q}^{\times}$. If an equation fails to have a solution in $\bZ_{p}^{\times}$ but does have a solution $(x_{1},\dots,x_{s})$ in the primes, then we must have $x_{j}=p$ for some $j$, which means that we have essentially lost a variable. This means that we expect fewer solutions, but not necessarily none. Since the circle method requires sufficiently many variables, we might not be able to apply it to these equations.

Another difference is that the Hasse principle is known to hold for all rational quadratic forms, but this is not true of the prime Hasse principle, as we will see in Theorem \ref{counterpythag} where we find infinitely many counterexamples which are diagonal in $3$ variables.

The question of, given $k$, finding the smallest $s$ such that \eqref{maineq} satisfies the prime Hasse principle for all $\ba$ is analogous to the Waring-Goldbach problem, in the same way that the Hasse principle for equations of the form \eqref{maineq} is analogous to Waring's problem. In \cite{kw2017}, Kumchev and Wooley show upper bounds for the numbers $H(k)$ for $k\ge 7$ in the Waring Goldbach problem, and we expect that their proof can be modified to give the same result for the prime Hasse principle, which would give the prime Hasse principle for equations of the form \eqref{maineq} in at least $4k(\log k +O(1))$ variables. We will not pursue this here, however.

Cook and Magyar considered something similar to the prime Hasse principle in \cite{cm2014} and showed that it holds for general algebraic varieties whenever a certain quantity (called the Birch rank) is sufficiently large. The solubility in prime variables of the equations $a_{1}x_{1}^{k}+\cdots+a_{s}x_{s}^{k}=n$ was considered in \cite{lt1991} ($k=2$) and in \cite{gl2020} ($k\ge 3$). In these papers it was shown that for $s\ge 2^{k}+1$, these equations have prime solutions, assuming some conditions which are a bit stronger than solubility in $\bR^{+}$ and $\bZ_{p}^{\times}$, and they also obtained an upper bound on the size of the smallest solution. We will also obtain upper bounds on the smallest solution in Theorem \ref{main}. These are of the same shape as those obtained in \cite{bd2014} for the integer case, and are better than those obtained in \cite{lt1991} and \cite{gl2020}, but at the cost of only holding for almost all equations. More recently Talmage in a preprint \cite{t2022} considered systems of equations of the form
\begin{align*}
a_{1,1}x_{1}^{k_{1}}+.&..+a_{s,1}x_{s}^{k_{1}}=0\\
a_{1,2}x_{1}^{k_{2}}+.&..+a_{s,2}x_{s}^{k_{2}}=0\\
&\vdots\\
a_{1,t}x_{1}^{k_{t}}+.&..+a_{s,t}x_{s}^{k_{t}}=0
\end{align*}
where $k_{1}<\cdots<k_{t}$. They showed that the Hasse principle always holds when the number of variables is at least $2s_{\eps}(\bk)+1$, where $s_{\eps}(\bk)$ depends on the strength of bounds on a certain mean value. When $t=1$, we have a single equation of the form \eqref{maineq}, and if one uses mean value bounds coming from Vinogradov's mean value theorem, then we get $s_{\eps}(k)\le (k^{2}+k)/2$ and so the prime Hasse principle is satisfied whenever $s\ge k^{2}+k+1$.

Though we only consider $k\ge2$, we note that the prime Hasse principle always holds for a single equation when $k=1$. When $s\le 2$, this is trivial, and when $s\ge3$ it can be deduced from \cite[Theorem 1.8]{gt2010}. This result also implies the prime Hasse principle for systems of linear equations under certain mild assumptions. We also have an upper bound on the smallest solution of $|\ba|^{c}$ for some constant $c$. This was originally shown by Liu and Tsang in \cite{lt1989}, with $c$ being subsequently improved by various authors, with the current best result being $c\le 25$ by Ching and Tsang \cite{ct2017}.

Define $\mathcal{C}'(k,s)$ to be the set of all $\ba$ such that \eqref{maineq} has a solution in $\bR^{+}$ and in $\bZ_{p}^{\times}$ for all $p$. Our main result is the following:
\begin{thm}\label{main}
Let $k\ge 2$ and $s$ be fixed integers and let $\hat{s}$ be the largest even number with $\hat{s}<s$. Suppose $\hat{s}\ge 3k$. Let $C=2^{1/(\hat{s}-k)}$. Then the number of $\ba\in \mathcal{C}'(k,s)$ with $|\ba|\le A$ such that \eqref{maineq} does not have a solution $(x_{1},\dots,x_{s})$ with $x_{j}$ prime and $x_{j}\le C|\ba|^{1/(\hat{s}-k)}$ for all $1\le j\le s$, is bounded above by
\[O_{\lambda}\left(\frac{A^{s}}{(\log A)^{\lambda}}\right)=o(A^{s})\]
for all $\lambda>0$. In particular, the prime Hasse principle \ref{pHp} is satisfied for almost all equations of the form \eqref{maineq} when $\hat{s}\ge 3k$, with an upper bound on the size of the smallest solution which almost always holds.

Furthermore, if we let $\lambda>0$ and $1/(\hat{s}-k)\le \xi\le 1/(2k)$, $C_{1}(\xi,\psi)=1/\xi^{\psi}$ and $C_{2}(\xi)=2^{\xi}$. Then there exists some $\psi_{0}(\lambda)>0$, such that for all $\psi\ge \psi_{0}(\lambda)$, the number of $\ba\in \mathcal{C}'(k,s)$ with $|\ba|\le A$ such that \eqref{maineq} does not have a solution $(x_{1},\dots,x_{s})$ with $x_{j}$ prime and $C_{1}|\ba|^{\xi}(\log |\ba|)^{-\psi}<x_{j}\le C_{2}|\ba|^{\xi}$ for all $1\le j\le s$, is bounded above by
\[O_{\lambda,\psi}\left(\frac{A^{s}}{(\log A)^{\lambda}}\right)=o(A^{s}).\]

In particular, we also get a lower bound for the largest solution, which almost always holds (although note that this lower bound depends on $\lambda$).
\end{thm}
To summarise, Principle \ref{pHp} is satisfied by \eqref{maineq} for almost all $\ba$ as long as $s\ge s_{0}(k)$, where
\[s_{0}(k)=\begin{cases}3k+2\:\text{if $k$ is odd}\\3k+1\:\text{if $k$ is even}\end{cases}.\]
The number of variables in theorem \ref{main} improves on the $k^{2}+k+1$ result we got from \cite{t2022} and the $2^{k}+1$ result from \cite{lt1991} and \cite{gl2020} when $k\ge 4$, and for $k$ sufficiently large improves on the $4k(\log k +O(1))$ bound which we could possibly obtain from modifying \cite{kw2017}. For $k\le 3$, the number of variables is worse than $2^{k}+1$, and for $k=4$ it is also possible that one could adapt a result of Zhao \cite{z2014} on the Waring-Goldbach problem to show that the prime Hasse principle holds for \emph{all} equations of the form \eqref{maineq} in $s_{0}(4)=13$ variables, rather than almost all as we show. However, we have chosen to state the theorem anyway for these cases because the bounds on the smallest solution still appear to be new.

We introduce some notation. We write $p$ or $p_{j}$ to always mean a prime number when occurring as a variable, unless stated otherwise. We also assume that any implied constants are allowed to depend on $s,k,\eps$ unless stated otherwise. We also use the notation $e(x)=e^{2\pi ix}$ and $e_{q}(x)=e^{2\pi ix/q}$ for $x\in \bR$, $q\in \bZ^{+}$. The following definitions make sense for all $k\ge 1$, although in the statements of theorems we will assume $k\ge2$ unless stated otherwise.

Given $B>0$ we define
\begin{equation}\label{gdef}
g(\alpha)=g(\alpha,B)=\sum_{p\le B} e\left( \alpha p^{k}\right) \log p
\end{equation}
and
\[\tilde{g}(\alpha)=\tilde{g}(\alpha,B,\psi)=\sum_{B(\log B)^{-\psi}<p\le B} e\left( \alpha p^{k}\right) \log p.\]
Given a tuple $\ba$ and a function $f$, we use the notation $f_{\ba}$ to mean
\[f_{\ba}(x)=\prod_{j} f(a_{j}x).\]
We also define $\rho_{\ba}(B)$ to be the log-weighted count of prime solutions to $\eqref{maineq}$:
\begin{equation}\label{rhodef}
\rho_{\ba}(B)=\sum_{\substack{a_{1}p_{1}^{k}+\cdots+a_{s}p_{s}^{k}=0\\ p_{j}\le B}} (\log p_{1})\cdots(\log p_{s})=\int_{0}^{1} g_{\ba}(\alpha)d\alpha.
\end{equation}

Note that if we can show $\rho_{\ba}(B)>0$ then $\eqref{maineq}$ will have a solution in primes $x_{j}\le B$. We also similarly define $\tilde{\rho}$, which will be used for the second part of Theorem \ref{main}:
\begin{equation}
\tilde{\rho}_{\ba}(B)=\sum_{\substack{a_{1}p_{1}^{k}+\cdots+a_{s}p_{s}^{k}=0\\ B(\log B)^{-\psi}<p_{j}\le B}} (\log p_{1})\cdots(\log p_{s})=\int_{0}^{1} \tilde{g}_{\ba}(\alpha)d\alpha.
\end{equation}
We define the singular series $\fS_{\ba}$ and the singular integral $J_{\ba}$ by
\begin{equation}\label{ssdef}
\fS_{\ba}=\sum_{q=1}^{\infty} T_{\ba}(q),
\end{equation}
where
\[T_{\ba}(q)=\varphi(q)^{-s}\sum_{\substack{r=1\\(r;q)=1}}^{q}\prod_{j}W(q,a_{j}r),\]
\[W(q,r)=\sum_{\substack{h=1 \\ (h;q)=1}}^{q} e_{q}\left( rh^{k}\right),\]
and
\begin{equation}\label{sidef}
J_{\ba}=\int_{-\infty}^{\infty} v_{\ba}(\beta) d\beta,
\end{equation}
where
\[v(\beta)=\int_{0}^{1}e\left(\beta x^{k}\right)dx.\]
It will turn out that $J_{\ba}\ne 0$ if and only if the $a_{j}$ do not all have the same sign, which is equivalent to \eqref{maineq} having a solution in $\bR^{+}$, and that $\fS_{\ba}\ne 0$ if and only if \eqref{maineq} has a solution in $\bZ_{p}^{\times}$ for each $p$ (see Remark \ref{singsernonzero} and the second part of Theorem \ref{main3}). Our aim is to show that the mean square error in the approximation of $\rho_{\ba}$ by $\fS_{\ba}J_{\ba}B^{s-k}$ is sufficiently small (Theorem \ref{main2}) and that if the singular series and singular integral are nonzero then they are almost always not too small (Theorem \ref{main3}).
\begin{thm}\label{main2}
Let $k\ge 2$, $\hat{s}\ge 3k$ and $\gamma>0$. Then for all \\$2\le B^{2k}\le A\le B^{\hat{s}-k}$:
\[\sum_{|\ba|\le A}\left| \rho_{\ba}(B)-\fS_{\ba}J_{\ba}B^{s-k}\right|^{2}\ll_{\gamma} \frac{A^{s-2}B^{2s-2k}}{(\log B)^{\gamma}},\]
and for $\psi$ sufficiently large in terms of $\gamma$:
\[\sum_{|\ba|\le A}\left| \tilde{\rho}_{\ba}(B)-\fS_{\ba}J_{\ba}B^{s-k}\right|^{2}\ll_{\gamma,\psi} \frac{A^{s-2}B^{2s-2k}}{(\log B)^{\gamma}}.\]
\end{thm}
\begin{thm}\label{main3}
Suppose $k\ge 2$, $s\ge 6$. Let $\eta>0$, then there exists $\delta\asymp \eta$ such that
\begin{equation}\label{main3eq1}\#\{|\ba|\le A: 0<\fS_{\ba}<(\log A)^{-\eta}\}\ll_{\eta,\delta}\frac{A^{s}}{(\log A)^{\delta}}=o\left(A^{s}\right).\end{equation}
Furthermore, if we suppose $s\ge k+1$, then for all $\ba$, $J_{\ba}\ne 0$ if and only if $a_{j}$ do not all have the same sign, and
\begin{equation}\label{main3eq2}\#\{|\ba|\le A: 0<J_{\ba}<|\ba|^{-1}(\log A)^{-\eta}\}\ll_{\eta,\delta}\frac{A^{s}}{(\log A)^{\delta}}=o\left(A^{s}\right).\end{equation}
\end{thm}

Note that if $\ba\notin \mathcal{C}'(k,s)$ (or equivalently if $\fS_{\ba}J_{\ba}=0$ when $s\ge \max\{6,k+1\}$) then the prime Hasse principle \ref{pHp} holds trivially. Hence, if it were the case that $\ba\in \mathcal{C}'(k,s)$ only for $o(A^{s})$ values of $|\ba|\le A$ then Theorem \ref{main} would be vacuous. However, we will show that, when $s\ge 4$, this is not the case, and in fact we will calculate the exact density of $|\ba|\le A$ for which $\ba\in\mathcal{C}'(k,s)$. Finding an exact formula for the density turns out to be easier than in the case of the integer Hasse principle. This is because if $a_{1},\dots,a_{s}$ are coprime, then for sufficiently large $n$, a solution in $(\bZ/p^{n}\bZ)^{\times}$ can be lifted to one in $\bZ_{p}^{\times}$ (see Lemma \ref{singserlemma1}). Solubility is then determined by the congruence classes of $\ba$ (mod $p^{n}$). On the other hand, lifting a solution from $\bZ/p^{n}\bZ$ to $\bQ_{p}$ is more difficult as some of the variables may be divisible by $p$.
\begin{thm}\label{main4}
Let $k\ge 2$, $s\ge 4$. For a prime $p$, let $C_{p}(k,s)$ be the set of $\ba$ (with $a_{j}\ne 0$ for all $j$) such that equation \eqref{maineq} has a solution in $\bZ_{p}^{\times}$. Define
\[\delta_{p}=\lim_{A\rightarrow\infty}\frac{\#\{|\ba|\le A: \ba\in C_{p}(k,s)\}}{(2A)^{s}}\]
and $\delta_{\infty}=1-2^{1-s}$. Then $\delta_{\infty}$ and $\delta_{p}$ are well-defined and positive, and are the densities of $\ba$ such that equation \eqref{maineq} has a solution in respectively $\bR^{+}$ and $\bZ_{p}^{\times}$. Also let
\[\fD_{s,k}=\delta_{\infty}\prod_{p}\delta_{p}.\]
Then this product converges absolutely and
\begin{equation}\label{density_limit}\#\{\ba\in \mathcal{C}'(k,s):|\ba|\le A\}=\fD_{s,k}(2A)^{s}+O\left(\frac{A^{s}}{(\log A)^{s-3}}\right).\end{equation}
Furthermore, we have, for sufficiently large $p$:
\begin{equation}\label{deltaformula}\delta_{p}=1-\left(1-\frac{1}{p^{s}}\right)^{-1}\left(s\left(1-\frac{1}{p}\right)\frac{1}{p^{s-1}}+\frac{s(s-1)}{2}\left(1-\frac{1}{p}\right)^{2}\left(1-g_{p,k}^{-1}\right)\frac{1}{p^{s-2}}\right),\end{equation}
where $g_{p,k}=(p-1;k)$.
\end{thm}
\begin{remark}
More precisely, \eqref{deltaformula} holds whenever $p$ is such that $p>2$, $p\nmid k$ and for every $\ba$, if at least $3$ of the $a_{j}$ are not divisible by $p$, then $\ba\in C_{p}(k,s)$. It will be shown in Lemma \ref{chinonzero} that these conditions hold for all sufficiently large $p$ and we will even do this explicitly.
\end{remark}
One might ask whether this is still true when $s=3$. Browning and Dietmann \cite{bd2009} proved that when $k\ge 2$, $s=3$, the set of $\ba$ which have a nonzero solution in $\bQ$ has density zero, and their proof also implies the stronger result that the set of $\ba$ which are soluble in $\bR$ and $\bQ_{p}$ for all $p$ has density 0. Koymans, Paterson, Santens and Shute made this explicit by finding an asymptotic formula \cite[Theorem 1.1]{kpss2025} for the number of such $\ba$. It follows that almost all $\ba$ do not lie in $\mathcal{C}'(k,3)$ for $k\ge 2$.

It should be noted at this point one further difference between the Hasse principle and prime Hasse principle. When $s$ is sufficiently large in terms of $k$, \eqref{maineq} always has nontrivial solutions in $\bQ_{p}$ for all $p$, and also satisfies the Hasse principle (see \cite[Theorems 1 and 2]{dl1963}), so that as long as there are nontrivial solutions in $\bR$, the equation always has nontrivial integer solutions. It is not true, however, that for sufficiently large $s$, there is always a solution to $\eqref{maineq}$ in $\bZ_{p}^{\times}$. One obstacle is that if $(p-1)\mid k$ then by Fermat's little theorem, $x^{k}\equiv 1$ (mod $p$) for all $x\in \bZ$ with $p\nmid x$. Therefore, no matter how large $s$ gets, by reducing (mod $p$) we see that there can only be a solution in $\bZ_{p}^{\times}$ if $a_{1}+\cdots+a_{s}\equiv 0$ (mod $p$). It follows that when $(p-1)\mid k$,
\begin{equation}\label{deltapwhenpbad}
\delta_{p}\le 1/p,
\end{equation}
where $\delta_{p}$ is as in Theorem \ref{main4}.

However, even if $a_{1}+\cdots+a_{s}\not\equiv 0$ (mod $p$), there might be a solution in the primes if $x_{j}=p$ for some $j$. For example if $k=2$ and $p=3$ then $3-1\mid 2$ but $x^{2}+13y^{2}-z^{2}=0$ has solution $x=2$, $y=3$, $z=11$, even though $1+13-1\equiv 1$ (mod $3$).

One might think then that there is still hope that for sufficiently large $s$, if \eqref{maineq} has solutions in $\bR^{+}$ then it will always have a solution in the primes. Unfortunately, this is not the case either, and in fact the number of counterexamples has nonzero density. Let $s\ge 2$ and $p$ a prime. If $p\nmid a_{i}$ for some $i$ but $p\mid a_{j}$ for all $j\ne i$, then clearly any solution to \eqref{maineq} must have $x_{i}=p$. Therefore, if we let $p_{1},\dots,p_{s}$ be the first $s$ primes, and take $a_{j}=b_{j}\prod_{\substack{1\le i\le s\\i\ne j}} p_{i}$ where $b_{j}$ are any nonzero integers with $p_{i}\nmid b_{j}$ for all $i,j$ then the only possible candidate for a solution is $x_{j}=p_{j}$ for all $j$. If we further suppose that the $a_{j}$ do not all have the same sign then there is a solution in $\bR^{+}$, and if we suppose that $\sum_{j} b_{j}\ne 0$ then $x_{j}=p_{j}$ can not be a solution and we therefore have no prime solutions. It is not hard to show that the set of $\ba$ described above has positive density. Perhaps we can at least show that the density of prime-soluble equations tends to $1$ as $s\rightarrow \infty$? The answer turns out to be yes.
\begin{thm}\label{main5}
Let $k\ge 2$. Then for some $\lambda(k)$,
\[\liminf_{A\rightarrow \infty}\frac{\#\{|\ba|\le A: \text{\eqref{maineq} has a solution in the primes} \}}{(2A)^{s}}=1-O\left(s^{\lambda(k)}2^{-s}\right),\]
where the implied constant depends only on $k$ (not on $s$).

In other words, the (lower) density of the tuples $\ba$ for which \eqref{maineq} has a solution in the primes tends to $1$ as $s\rightarrow \infty$
\end{thm}

We showed earlier that the converse to the prime Hasse principle \ref{pHp} does not always hold, but it turns out that there is a partial converse for equations of the form \eqref{maineq} which almost always holds.
\begin{thm}\label{main6}
Fix some $\lambda>0$ and consider the following partial converse to the prime Hasse principle:

\begin{enumerate}\item[$(*)$]If equation \eqref{maineq} has a solution $\bx=(x_{1},\dots,x_{s})$ with $x_{j}$ prime and $x_{j}> |\ba|^{\lambda}$ for all $j$ then $\ba\in\mathcal{C}'(k,s)$.\end{enumerate}

Then for $k\ge 2$, $s\ge 3$, $\lambda\ge 1$ and $|\ba|$ sufficiently large in terms of $k$, $(*)$ always holds. For $k\ge 2$, $s\ge 4$ and  all $\lambda>0$, we have:
\[\#\{|\ba|\le A: (*)\text{ does not hold}\}\ll_{\lambda} A^{s-\lambda(s-3)}=o(A^{s}),\]
so $(*)$ almost always holds.
\end{thm}

Theorem \ref{main} and Theorem \ref{main6} together imply the following. If $\hat{s}\ge 3k$ and $\lambda< 1/(2k)$, then for almost all $\ba$, equation \eqref{maineq} has a solution in primes $x_{j}$ with $x_{j}>|\ba|^{\lambda}$ for all $j$ if and only if $\ba\in \mathcal{C}'(k,s)$.

In the final section, we will give some counterexamples to the prime Hasse principle \ref{pHp}. In particular, we will show that for each $k\ge2$ there are infinitely many diagonal equations in $3$ variables which do not satisfy the prime Hasse principle. When $k=2$ this means that the analogue of the Hasse-Minkowski theorem is not true.

\addtocontents{toc}{\protect\setcounter{tocdepth}{0}}
\section*{Acknowledgements and funding}
The author would like to thank their PhD supervisor Sam Chow for his help and advice in writing this paper.

Philippa Holdridge is supported by the Warwick Mathematics Institute Centre for Doctoral Training, and gratefully acknowledges funding from the University of Warwick.

\section*{Organisation and Notation}
In Section \ref{singsersection} we will prove results about the singular series and singular integral, including Theorem \ref{main3}, by adapting the methods of \cite[Section 3]{bd2014}. Then, in Section \ref{geometrynumbersection} we prove an analogue of \cite[Lemma 2.5]{bd2014} for primes using the geometry of numbers. The methods of this section are largely the same as in that paper, adapted to deal with prime variables, but in some places, they will actually be simpler because the contribution from linearly independent pairs of solutions is much easier to bound. In Section \ref{arcsection} we prove Theorem \ref{main2} by the circle method, with the main result of Section \ref{geometrynumbersection} being used to give a mean value bound on the minor arcs. We will then deduce Theorem \ref{main} from Theorems \ref{main2} and \ref{main3}.

We then move on to proving the remaining theorems. In Section \ref{localdensitysection} we will investigate local solubility and prove Theorem \ref{main4}. In Section \ref{primesolublesection} we will prove Theorem \ref{main5}. Finally, in Section \ref{countersection} we will prove Theorem \ref{main6} as well as give some counterexamples to the prime Hasse principle and show in particular that there are infinitely many counterexamples of degree $2$.

We use $\eps$ to mean an abritrary positive constant (generally thought of as small). We use the notation $f(x)\ll g(x)$ to mean $f(x)=O(g(x))$ and $f(x)\asymp g(x)$ to mean that $f(x)\ll g(x)$ and $g(x)\ll f(x)$ both hold. Implied constants in $O$, $\ll$ and $\asymp$ notations will always be allowed to depend on $\eps$ and on $s$, $k$ unless stated otherwise, and we will denote dependence on other parameters with a subscript. For example, $O_{\lambda}$ means that the constant may depend on $\lambda$.

Given a prime $p$ and integers $n\ge 0$ and $a$, we use the notation $p^{n}\Vert a$ to mean that $p^{n}\mid a$ and $p^{m}\nmid a$ for all larger $m$.

\addtocontents{toc}{\protect\setcounter{tocdepth}{1}}
\section{The singular series and singular integral}\label{singsersection}

\subsection{The singular series}
In this subsection we prove a product formula for the singular series. We will see that $\fS_{\ba}\ne 0$ if and only if equation \eqref{maineq} has solutions in $\bZ_{p}^{\times}$ for all $p$. First, let $M_{\ba}(p^{n})$ be the number of solutions to
\begin{equation}\label{mainmod}
a_{1}x_{1}^{k}+\cdots+a_{s}x_{s}^{k}\equiv 0 \left(\text{mod }p^{n}\right)
\end{equation}
where $x_{j}\in (\bZ/p^{n}\bZ)^{\times}$. The $p$-adic density is then:
\begin{equation}\label{paddensdef}
\chi_{p}=\chi_{p}(\ba)=\lim_{n\rightarrow \infty} p^{n}\varphi\left(p^{n}\right)^{-s}M_{\ba}\left(p^{n}\right).
\end{equation}
It is not obvious that the limit exists, but we will show that this is the case.
\begin{lemma}\label{singprod}
If $s\ge 3$, then $\chi_{p}\ge 0$ is well-defined. If $s\ge 5$, then $\fS_{\ba}\ge 0$ is well-defined, and
\[\fS_{\ba}=\prod_{p} \chi_{p},\]
where the product converges absolutely, so that $\fS_{\ba}>0$ if and only if $\chi_{p}>0$ for all $p$.
\end{lemma}
We will prove this in a series of steps.
\begin{lemma}\label{Tazero}
Let $p$ be a prime. Define $\nu(p)$ by $p^{\nu(p)}\Vert k$ and
\begin{equation}\label{xidef}\xi(p)=\begin{cases}\nu(p)+2\quad\text{if $p=2$ and $2\mid k$}\\\nu(p)+1\quad\text{otherwise}\end{cases}.\end{equation}
Then for all $n> \xi(p)$ and all $r\in \bZ$ with $p\nmid r$:
\[W(p^{n},r)=0.\]
Hence, if $p\nmid (a_{1};\dots;a_{s})$ and $n> \xi(p)$, then:
\[T_{\ba}(p^{n})=0.\]
\begin{proof}
The first part is \cite[Lemma 8.3]{hua}, and the second part follows immediately by considering the definition of $T_{\ba}$.
\end{proof}
\end{lemma}
\begin{lemma}\label{singconv}
We have:
\begin{equation}\label{Wbound}
W(q,r)\ll q^{1/2+\eps}(q;r)^{1/2}
\end{equation}
and hence
\begin{equation}\label{Tabound}
T_{\ba}(q)\ll q^{1-s/2+\eps}\prod_{j}(a_{j};q)^{1/2}.
\end{equation}
Furthermore, if $s\ge 5$ then the series for $\fS_{\ba}$ \eqref{ssdef} converges absolutely, so in particular $\fS_{\ba}$ is well-defined.
\begin{proof}
Applying \cite[Lemma 8.5]{hua} gives \eqref{Wbound} in the case $(r;q)=1$. For the general case we let $q'=q/(r;q)$ and $r'=r/(r;q)$. Then by periodicity:
\[W(q,r)=\sum_{\substack{h=1 \\ (h;q)=1}}^{q} e_{q}\left( rh^{k}\right)=(r;q)\sum_{\substack{h=1 \\ (h;q)=1}}^{q'} e_{q'}\left( r'h^{k}\right).\]
This is almost $(r;q)W(q',r')$ but for the condition of $(h;q)=1$ instead of $(h;q')=1$ in the sum. To avoid this problem, we first restrict to the case $q=p^{n}$. Let $p^{\theta}\Vert r$. Note that if $\theta\ge n$ then $W(q,r)=\varphi(q)\le q=(q(r;q))^{1/2}$, so we may assume that $\theta<n$. We then get:
\[W(q,r)=(r;q)W(q',r')\ll (r;q) (q')^{1/2+\eps}=(r;q)\left( \frac{q}{(r;q)}\right)^{1/2+\eps}\le q^{1/2+\eps}(r;q)^{1/2}.\]
It remains to consider the case where $q$ is not a prime power. This will be dealt with using \cite[Lemma 8.1]{hua}, which states that if $(q_{1},q_{2})=1$ then
\begin{equation} \label{hua_lemma8.1}W(q_{1}q_{2},r)=W\left( q_{1},rq_{2}^{k-1}\right) W\left( q_{2},rq_{1}^{k-1}\right).\end{equation}
Let $C$ be a constant such that $W(p^{n},r)\le Cp^{n(1/2+\eps)}(r;p^{n})^{1/2}$ for all $p,n,r$ and suppose $q=p_{1}^{e_{1}}\cdots p_{m}^{e_{m}}$. Then by applying \eqref{hua_lemma8.1} and noting that $(rq_{2}^{k-1};q_{1})=(r;q_{1})$, we have
\[W(q,r)\le C^{m}q^{1/2+\eps}(q;r)^{1/2}.\]
The result \eqref{Wbound} now follows from the fact that $C^{m}\le\tau(q)^{\log_{2} C}\ll q^{\eps}$, where $\tau$ is the divisor function, and \eqref{Tabound} follows from \eqref{Wbound} and the fact that $\varphi(q)\gg q^{1-\eps}$. The last part then follows from \eqref{Tabound}.
\end{proof}
\end{lemma}
\begin{lemma}\label{chisum}
We have the identity:
\[M_{\ba}(p^{n})=p^{-n}\varphi\left(p^{n}\right)^{s}\sum_{m=0}^{n} T_{\ba}\left( p^{m}\right).\]
Hence, for $s\ge 3$, $\chi_{p}$ is well-defined and
\begin{equation}\label{chisumformula}\chi_{p}=\sum_{m=0}^{\infty} T_{\ba}\left( p^{m}\right).\end{equation}
\begin{proof}
We use the fact that
\[q^{-1}\sum_{n=1}^{q} e_{q}(mn)=\begin{cases}
1 &\text{if }m\equiv 0\text{ (mod $q$)}\\
0 &\text{if }m\not\equiv 0\text{ (mod $q$)}\end{cases}.\]
We get:
\begin{align*}
M_{\ba}\left(p^{n}\right)&=\sum_{x_{1},\dots,x_{s}\in (\bZ/p^{n}\bZ)^{\times}}p^{-n}\sum_{m=0}^{p^{n}-1} e_{p^{n}}\left(m\left(a_{1}x_{1}^{k}+\cdots+a_{s}x_{s}^{k}\right)\right)\\&=p^{-n}\sum_{m=0}^{p^{n}-1}\prod_{j}W\left(p^{n},a_{j}m\right).
\end{align*}
Now we write $m=p^{e}m'$ with $p\nmid m'$ and split the sum according to the value of $e$. Noting that
\[W\left(p^{n},a_{j}p^{e}m'\right)=\begin{cases}p^{e}W\left(p^{n-e},a_{j}m'\right)&\text{if }e<n\\ p^{n-1}(p-1)& \text{if }e=n\end{cases},\]
we get
\[M_{\ba}\left(p^{n}\right)=p^{-n}\varphi\left(p^{n}\right)^{s}\sum_{l=0}^{n} T_{\ba}\left(p^{l}\right).\]
Then \eqref{chisumformula} follows, with convergence of the sum for $s\ge 3$ by \eqref{Tabound}.
\end{proof}
\end{lemma}
\begin{lemma}\label{sollift}
If $s\ge 3$ and $p\nmid (a_{1};\dots;a_{s})$, then for all $l\ge \xi(p)$,
\[\chi_{p}=\sum_{m=0}^{l}T_{\ba}\left(p^{m}\right)=\sum_{m=0}^{\xi(p)}T_{\ba}\left(p^{m}\right)=p^{\xi(p)}\varphi\left(p^{\xi(p)}\right)^{-s}M_{\ba}\left(p^{\xi(p)}\right).\]
Hence, if there exists a solution in $(\bZ/p^{\xi(p)}\bZ)^{\times}$, then $\chi_{p}\ge p^{-(s-1)\xi(p)}>0$.
\begin{proof}
This follows easily from Lemmas \ref{Tazero} and \ref{chisum}.
\end{proof}
\end{lemma}
There is one more lemma needed to deduce Lemma \ref{singprod}.
\begin{lemma}\label{Tamult}
$T_{\ba}$ is a multiplicative function.
\begin{proof}
Let $q_{1},q_{2}$ be positive integers with $(q_{1};q_{2})=1$. Recall \cite[Lemma 8.1]{hua}, which gives us
\[W(q_{1}q_{2},a_{j}r)=W\left(q_{1},a_{j}rq_{2}^{k-1}\right) W\left(q_{2},a_{j}rq_{1}^{k-1}\right).\]
Now by the Chinese remainder theorem we have:
\begin{align*}
\sum_{\substack{r=1\\ (r;q)=1}}^{q} \prod_{j} &W(q_{1}q_{2},a_{j}r)\\ &=\left( \sum_{\substack{r_{1}=1\\(r_{1};q_{1})=1}}^{q_{1}}\prod_{j} W\left( q_{1}, a_{j}r_{1}q_{2}^{k-1}\right)\right) \left( \sum_{\substack{r_{2}=1\\(r_{2};q_{2})=1}}^{q_{2}} \prod_{j}W\left( q_{2}, a_{j}r_{2}q_{1}^{k-1}\right)\right)\\ &=\left( \sum_{\substack{r_{1}=1\\(r_{1};q_{1})=1}}^{q_{1}} \prod_{j}W\left( q_{1}, a_{j}r_{1}\right)\right)\left( \sum_{\substack{r_{2}=1\\(r_{2};q_{2})=1}}^{q_{2}} \prod_{j}W\left( q_{2}, a_{j}r_{2}\right)\right).
\end{align*}
Combining this with the fact that $\varphi$ is multiplicative gives that $T_{\ba}$ is multiplicative.
\end{proof}
\end{lemma}
\begin{proof}[Proof of Lemma \ref{singprod}]
When $s\ge 3$, existence of $\chi_{p}$ follows from Lemma \ref{chisum}.

Suppose now that $s\ge 5$. The sums $\sum_{q}T_{\ba}(q)$ and $\sum_{p}\sum_{k=1}^{\infty}T_{\ba}(p^{k})$ both converge absolutely by Lemma \ref{singconv}. Combining this with multiplicativity of $T_{\ba}$ (Lemma \ref{Tamult}), it follows from \cite[Theorem 4.6]{koukoulopoulos} that
\[\prod_{p} \sum_{m=0}^{\infty} T_{\ba}\left(p^{m}\right)=\sum_{q=1}^{\infty} T_{\ba}(q)=\fS_{\ba}.\]
By Lemma \ref{chisum}, the left hand side is equal to $\prod_{p}\chi_{p}$. It just remains to show absolute convergence of the product, which is equivalent to
\[\sum_{p} |\chi_{p}-1|<\infty.\]
By absolute convergence of the singular series for $s\ge 5$ (Lemma \ref{singconv}):
\[\sum_{p}|\chi_{p}-1|\le \sum_{p}\sum_{n=1}^{\infty} \left|T_{\ba}\left(p^{n}\right)\right|\le \sum_{q=1}^{\infty} \left| T_{\ba}(q)\right|<\infty.\]
\end{proof}
Now we have our product formula, we can find a lower bound for $\fS_{\ba}$ by bounding each $\chi_{p}$. Let $\mathcal{S}(\ba)$ be the set of all primes which divide at least $s-4$ of the coefficients. Then we have the following:
\begin{lemma}\label{chisandwich}
Let $k\ge2$, $s\ge 6$. Then for $p\notin \mathcal{S}(\ba)$ we have, for some constant $c$ independent of $s,\ba$,
\begin{equation}\label{chibound}
1-cp^{-3/2+\eps}\le\chi_{p}\le1+cp^{-3/2+\eps}
\end{equation}
and hence there exists $C$ independent of $s,\ba$ such that
\begin{equation}\label{chiprodbound}
\frac{1}{2}\le \prod_{\substack{p>C\\p\notin \mathcal{S}(\ba)}}\chi_{p}\le 2.
\end{equation}
\begin{proof}
Let $k\ge2$. Then $p\notin \mathcal{S}(\ba)$ implies that at least $5$ of the $a_{j}$ are not divisible by $p$, so by equation \eqref{Tabound}, we get
\[T_{\ba}\left(p^{n}\right)\ll p^{(-3/2+\eps)n}.\]
We also have $T_{\ba}(1)=1$. Equation \eqref{chibound} then follows by summing a geometric series. Now the absolute convergence of the products $\prod_{p} (1\pm cp^{-3/2+\eps})$ implies \eqref{chiprodbound}.
\end{proof}
\end{lemma}
We now let $C$ be as in Lemma \ref{chisandwich} and define $\mathcal{P}(\ba)=\mathcal{S}(\ba)\cup \{p:p\le C\}$ and
\begin{equation}\label{lambdadef}\lambda(p)=\min\{l:p^{l}\Vert a_{j} \text{ for some }j\},\end{equation}
\[P(\ba)=\prod_{p\in \mathcal{P}(\ba)}p,\quad P_{0}(\ba)=\prod_{\substack{p\in \mathcal{S}(\ba)\\p>C}} p,\quad H=\prod_{p\le C} p.\]
Note that also
\[(a_{1};\dots;a_{s})=\prod_{p\in \mathcal{P}(\ba)}p^{\lambda(p)}.\]
We aim to prove firstly that if $P(\ba)$ is small, then $\fS_{\ba}$ is not too small, and secondly that $P(\ba)$ is small for most $\ba$.
\begin{lemma}\label{singserlemma1}
Let $s\ge 3$. We have $\chi_{p}>0$ if and only if \eqref{maineq} has a solution in $\bZ_{p}^{\times}$, if and only if it has a solution in $(\bZ/p^{\xi(p)+\lambda(p)}\bZ)^{\times}$ and in this case we have
\[\chi_{p}\ge p^{-\xi(p)(s-1)+\lambda(p)}.\]

Further, suppose that $s\ge5$, $|\ba|\le A$, $\fS_{\ba}\ne 0$. Let $\eta>0$, $\delta=\eta/(2(s-1))$ and let $c>0$ be a sufficiently small constant. Then if $P(\ba)\le c(\log A)^{\delta}$ then $\fS_{\ba}\ge(\log A)^{-\eta}$.
\begin{proof}
From the definition of $\chi_{p}$, if $\chi_{p}>0$ then $M_{\ba}(p^{n})>0$ for sufficiently large $n$. This means that \eqref{maineq} has a solution in $(\bZ/p^{n}\bZ)^{\times}$ for sufficiently large $n$, and this can be lifted to a solution in $\bZ_{p}^{\times}$ by Hensel's lemma. If we have a solution in $\bZ_{p}^{\times}$, then this reduces (mod $p^{n}$) to give a solution in $(\bZ/p^{n}\bZ)^{\times}$ for every $n\ge 1$.

Now suppose that there is a solution in $(\bZ/p^{\xi(p)+\lambda(p)}\bZ)^{\times}$. When $p\nmid (a_{1};\dots;a_{s})$, Lemma \ref{sollift} says that $\chi_{p}\ge p^{-\xi(p)(s-1)}$, so suppose $p\mid (a_{1};\dots;a_{s})$. Let $b_{j}=a_{j}/p^{\lambda(p)}$. Then by Lemma \ref{sollift} and Lemma \ref{chisum}, for $l>\xi(p)$, we have $M_{\bb}(p^{l})\ge p^{(s-1)(l-\xi(p))}$. Also, for $l>\lambda(p)$ and $x_{j}\in \bZ$ with $p\nmid x_{j}$, we have
\[a_{1}x_{1}^{k}+\cdots+a_{s}x_{s}^{k}\equiv 0\: \left(\text{mod }p^{l}\right)\]
if and only if
\[b_{1}x_{1}^{k}+\cdots+b_{s}x_{s}^{k}\equiv 0\: \left(\text{mod }p^{l-\lambda(p)}\right).\]
Therefore, if $l>\lambda(p)+\xi(p)$ then: 
\[M_{\ba}(p^{l})=p^{s\lambda(p)}M_{\bb}(p^{l-\lambda(p)})\ge p^{s\lambda(p)+(s-1)(l-\xi(p)-\lambda(p))}.\]
So $\chi_{p}\ge p^{-\xi(p)(s-1)+\lambda(p)}$.

Now let $|\ba|\le A$ and suppose the stated hypotheses hold. We have (recalling the definition \eqref{xidef} of $\xi(p)$):
\begin{align*}
\fS_{\ba}&\ge \frac{1}{2}\prod_{p\in \mathcal{P}(\ba)} \chi_{p}\\&\ge \frac{1}{2}\left(\prod_{p|k} p^{-\nu(p)(s-1)}\right)2^{-(s-1)}\left(\prod_{p\in \mathcal{P}(\ba)}p^{-(s-1)}\right)\left(\prod_{p} p^{\lambda(p)}\right)\\&=2^{-s}k^{-(s-1)}P(\ba)^{-(s-1)}(a_{1};\dots;a_{s}).
\end{align*}
Recalling $\delta=\eta/(2(s-1))$ and choosing $c$ sufficiently small then gives the result.
\end{proof}
\end{lemma}
\begin{remark}\label{singsernonzero}
It follows from Lemma \ref{singprod} and Lemma \ref{singserlemma1} that $\fS_{\ba}>0$ if and only if \eqref{maineq} has a solution in $\bZ_{p}^{\times}$ for all $p$.
\end{remark}
\begin{lemma}
Suppose $k\ge2$, $s\ge 6$. For all $\delta>0$, $c>0$, let
\[\mathcal{A}=\{|\ba|\le A:P(\ba)>c(\log A)^{\delta}\}.\]
Then for all $0<\eps<1$,
\[\#\mathcal{A}\ll_{c,\eps} \frac{A^{s}}{(\log A)^{\delta(s-5-\eps)}}=o\left(A^{s}\right).\]
\begin{proof}
When $k\ge2$, by definition, $P(\ba)=P_{0}(\ba)H$. Note that if $p\in \mathcal{P}(\ba)$ and $p>C$ then $p\in \mathcal{S}(\ba)$ so $p$ divides at least $s-4$ of the $a_{j}$. Hence $P_{0}(\ba)^{s-4}|a_{1}\cdots a_{s}$. If $P_{0}(\ba)=m$, then there exists $m_{1},\dots,m_{s}$ such that $m_{1}\cdots m_{s}=m^{s-4}$ and $m_{j}|a_{j}$ for each $j$. Hence the number of $\ba$ with $P_{0}(\ba)=m$ is at most
\[\sum_{m_{1}\cdots m_{s}=m^{s-4}} \prod_{j} \left \lfloor \frac{A}{m_{j}}\right \rfloor\le \sum_{m_{1}\cdots m_{s}=m^{s-4}} \frac{A^{s}}{m^{s-4}} \ll \frac{A^{s}}{m^{s-4-\eps}}.\]
Now summing over $m$, and recalling that $s\ge 6$, gives
\[\#\mathcal{A}\ll \sum_{m>c(\log A)^{\delta}/H} \frac{A^{s}}{m^{s-4-\eps}}\ll c^{-s+5+\eps}A^{s}(\log A)^{-\delta(s-5-\eps)}.\]
\end{proof}
\end{lemma}
We may now prove the first half of \ref{main3}.
\begin{proof}[Proof of Theorem \ref{main3} (part 1)]
Let $k\ge 2$, $\eta>0$ and $s\ge 6$. It follows immediately from the previous two lemmas that
\[\#\{|\ba|\le A: 0<\fS_{\ba}<(\log A)^{-\eta}\}\ll\frac{A^{s}}{(\log A)^{\delta'(s-5-\eps)}},\]
where $\delta'=\eta/(2(s-1))$. Now take $\eps=1/2$, $\delta=\delta'(s-5-\eps)$ and note that $\delta\asymp \eta$ as required.
\end{proof}
It remains to prove the part about $J_{\ba}$.
\subsection{The singular integral}
\begin{proof}[Proof of Theorem \ref{main3} (part 2)]
Recall \eqref{sidef}, which implies that
\[J_{\ba}=\int_{-\infty}^{\infty}v_{\ba}(\beta)d\beta=\int_{-\infty}^{\infty} \int_{0}^{1}\dots \int_{0}^{1}e\left(\left(a_{1}x_{1}^{k}+\cdots+a_{s}x_{s}^{k}\right)\beta\right)dx_{1}\cdots dx_{s}d\beta.\]
This is almost the same as the singular integral considered in \cite[Section 3.1]{bd2014}, but with $v(\beta)=\int_{0}^{1}e(\beta x^{k})dx$ instead of $\int_{-1}^{1}e(\beta x^{k})dx$. Assume without loss of generality that $|a_{s}|\ge |a_{j}|$ for $1\le j\le s-1$. The proof in \cite[Section 3.1]{bd2014} is split into two cases based on the parity of $k$. When $k$ is even, we have
\[\int_{-1}^{1}e(\beta x^{k})dx=2\int_{0}^{1}e(\beta x^{k})dx,\]
so it follows from the results of \cite[Section 3.1]{bd2014} that, when $k$ is even we have
\begin{equation}\label{JaintermsofE}J_{\ba}=k^{-s}|a_{1}\cdots a_{s}|^{-1/k}E(0),\end{equation}
where $E(\tau)$ is a certain integral over the region $\fC(\tau)$, which is the set of $(\eta_{1},\dots,\eta_{s-1})$ satisfying
\begin{align*}&0\le \eta_{j}\le |a_{j}|\quad (1\le j<s),\\&0\le \tau-\sigma_{s}\sigma_{1}\eta_{1}-\sigma_{s}\sigma_{2}\eta_{2}-\cdots-\sigma_{s}\sigma_{s-1}\eta_{s-1}\le |a_{s}|,\end{align*}
where $\sigma_{j}=a_{j}/|a_{j}|$.\footnote{In \cite{bd2014} it is stated without proof that $E(\tau)$ is of bounded variation near $0$. This can be proven using a method similar to that in the proof of \cite[Lemma 2.3]{c2014}.} We remark that it is assumed in \cite{bd2014} that $k\ge 3$, but the proof of \eqref{JaintermsofE} still holds when $k=2$.

Now if all the $a_{j}$ have the same sign then it is clear that $\fC(0)=\{(0,..,0)\}$ and hence $J_{\ba}=0$. Conversely, if the $a_{j}$ do not all have the same sign then we follow the rest of the proof from \cite[Section 3.1]{bd2014}. The authors of \cite{bd2014} claim that this leads to a bound of $J_{\ba}\gg |\ba|^{-1}$, but this is not true in general. In the case where $a_{s}$ has opposite sign to all of $a_{1},\dots,a_{s-1}$, the proof from \cite{bd2014} still goes through, but in the other case, we have, for certain $r,t \in \{1,\dots,s-1\}$,
\[E(0)\gg |a_{1}\cdots a_{t}|^{1/k}|a_{r-1}|^{(s-t-1)/k}|a_{r-1}|^{(1-k)/k}.\]
If we assume that $A \ge |a_{i}|\ge A(\log A)^{-\delta}$ for all $i$, then we still recover a bound of $J_{\ba}\ge A^{-1}(\log A)^{-O(\delta)}$, which gives the result.

Now for the odd case, if we follow the proof from \cite[Section 3.1]{bd2014} as we did in the even case, then we find that it is only the first step of replacing $\int_{-1}^{1}e(a_{j}\beta\xi^{k})d\xi$ with $2\int_{0}^{1}e(a_{j}\beta\xi^{k})dx$ which fails. However, in our case we start with $v(a_{j}\beta)=\int_{0}^{1}e(a_{j}\beta\xi^{k})d\xi$ and so we do not have this problem, and we get the same result for odd $k$ as well. This completes the proof of Theorem \ref{main3}\end{proof}

\section{Applications of the geometry of numbers}\label{geometrynumbersection}
The main result of this section will be to establish a mean value estimate that will help us with the minor arcs. The proof will be largely the same as that of Lemma 2.5 of \cite{bd2014}, but will in some places be simpler. As in section 2 of \cite{bd2014}, we will use facts about sublattices of $\bZ^{n}$. Given a sublattice $\Lambda$ of $\bZ^{n}$, we define its rank, its dual $\Lambda^{*}$ and the quantity $G(\Lambda)$ as in \cite[Section 2.1]{bd2014}. We define $\det \Lambda$ (the determinant of $\Lambda$) to be the volume of the parallelepiped spanned by a basis for $\Lambda$ (in \cite{bd2014}, this is called the discriminant and is denoted $d(\Lambda)$). We will also use \cite[Lemma 2.1]{bd2014}, which says that $\det\Lambda^{*}=(\det \Lambda)/G(\Lambda)$, as well as the following lemma.
\begin{lemma}\label{latticebound}
If $\Lambda\subset \bZ^{n}$ is a lattice of rank $r$ and $R\gg \det \Lambda$ then the ball of radius $R$ contains $O(R^{r}/\det \Lambda)$ points of $\Lambda$.
\begin{proof}
This is \cite[Lemma 1 (v)]{h-b2002}.
\end{proof}
\end{lemma}

The following lemma is similar to \cite[Lemma 2.4]{bd2014}, but is simpler. We define a semiprime to be a positive integer with exactly two prime factors. We also write $\omega(n)$ for the number of distinct prime factors of $n$, and define $\log_{(m)}$ inductively by $\log_{(1)} x = \max\{1,\log x\}$, $\log_{(m+1)} x=\max\{1,\log_{(m)} x\}$.
\begin{lemma}\label{semiprimepower}
Given $d\ge 1$, the number of semiprimes $x,y\le B$ such that
\[x^{k}\equiv y^{k}(\text{mod }d)\]
is bounded by
\[O\left(\tau(d)^{\lambda(k)}(\log_{(3)} d)\left(B+B^{2}d^{-1}\right)\right),\]
where $\lambda(k)=\log k/\log 2$.
\begin{proof}
Suppose first that $x$ is coprime to $d$. The key idea is to use the fact that for each (coprime) $k$-th power residue $\alpha$ (mod $d$), there are $O(k^{\omega(d)})=O(\tau(d)^{\lambda(k)})$ residues $\beta$ such that $\beta^{k}\equiv \alpha$ (mod $d$). This can be shown using the Chinese remainder theorem and the fact that $(\bZ/p^{n}\bZ)^{\times}$ is cyclic for odd $p$. So for each choice of $x$, there are $O(\tau(d)^{\lambda(k)})$ possibilities for $y$ (mod $d$), and hence $O(\tau(d)^{\lambda(k)}(1+Bd^{-1}))$ possible values for $y$. Summing over $x$ gives an upper bound of
\[B\tau(d)^{\lambda(k)}+B^{2}\tau(d)^{\lambda(k)}d^{-1}+R,\]
where $R$ is the number of solutions where $x,d$ are not coprime. For such a solution, there is a prime $p$ such that $p|d$, $p|x$. Then also $p|y$ and if we let $x'=x/p$, $y'=y/p$, $d'=d/p$ and $B'=B/p$ then
\[(x')^{k}\equiv (y')^{k}(\text{mod }d').\]
If $(x';d')>1$, then we must also have $(y';d')>1$, and as $x,y$ are semiprimes this means that $x|d$ and $y|d$. There are at most $\omega(d)^{4}\ll (\log d)^{4}$ possibilities for $x,y$ in this case. Without loss of generality, $d\le B^{k}$ because otherwise $x^{k}\equiv y^{k}$ (mod $d$)$\implies x=y$. So $(\log d)^{4}\ll B$. If $(x';d')=1$ then by the same argument as above there are $O(\tau(d')^{\lambda(k)}(B'+(B')^{2}(d')^{-1}))$ possible values for $x',y'$. Summing over all $p$ gives:
\[\sum_{p|d} \tau(d/p)^{\lambda(k)}\left(\frac{B}{p}+\frac{B^{2}}{p^{2}}\left(\frac{d}{p}\right)^{-1}\right)\ll \tau(d)^{\lambda(k)}\left(B+B^{2}d^{-1}\right)\sum_{p|d} \frac{1}{p}.\]
If we let $p_{k}$ be the $k$th prime. Then, noting that $\omega(d)\ll \log d$, we see that
\begin{equation}\label{reciprocal_div_bound}\sum_{p\mid d} \frac{1}{p}\le \sum_{k=1}^{\omega(n)}\frac{1}{p_{k}}\ll \log_{(2)} \omega(n) \ll \log_{(3)} n.\end{equation}
So we have $R\ll \tau(d)^{\lambda(k)}(\log_{(3)} d)(B+B^{2}d^{-1})$.
\end{proof}
\end{lemma}
\begin{lemma}\label{divisorsumofdivisor}
Let $n\ge 3$ be an integer and $r\ge 0$ a real number. Then there exists $C_{r}>0$ such that:
\[\sum_{d\mid n}\frac{\tau(d)^{r}}{d}\ll_{r} (\log_{(2)} n)^{C_{r}}.\]
\begin{proof}
Let
\[f(n)=\sum_{d\mid n}\frac{\tau(d)^{r}}{d}.\]
Note that $f$ is multiplicative. We compute $f(p^{k})$ for $k\ge 0$:
\[f\left(p^{k}\right)=\sum_{j=0}^{k}\frac{(j+1)^{r}}{p^{j}}=1+O_{r}\left(\frac{1}{p}\right).\]
Then for some $C'_{r}>0$,
\[f(n)\le \prod_{p\mid n}\left( 1+\frac{C'_{r}}{p}\right).\]
Taking logarithms:
\[\log f(n)\le C'_{r}\sum_{p\mid n} \frac{1}{p}.\]
Applying \eqref{reciprocal_div_bound}, we get
\[\log f(n)\ll_{r} \log_{(3)} n.\]
The result follows by exponentiating.
\end{proof}
\end{lemma}
Let $\Upsilon_{t}(A,B)=\sum \prod_{j} (\log p_{j})(\log q_{j})$, where the sum ranges over all the solutions to
\begin{equation}\label{doublemain}
\sum_{j=1}^{2t} a_{j}p_{j}^{k}=\sum_{j=1}^{2t} a_{j}q_{j}^{k}=0,
\end{equation}
where $p_{j},q_{j}$ are primes with
\[0<|a_{j}|\le A,\quad p_{j}\le B,\quad q_{j}\le B.\]
(Note that the $a_{j}$ are variables in this equation, unlike in \eqref{maineq}.)
\begin{lemma}\label{meanval}
Let $k\ge 1$, $t\in \bZ$ with $2t\ge \min\{4,k+2\}$ and suppose that $A\ge B^{2k}\ge 2$. Then there exists some $K(t,k)>0$ such that
\begin{equation}\label{Upsilonbound}\Upsilon_{t}(A,B)\ll (\log B)^{K(t,k)}A^{2t-2}B^{4t-2k}+(\log B)^{K(t,k)}A^{2t-1}B^{2t-k}.\end{equation}
\begin{remark}
By considering the underlying Diophantine equation and applying \cite[Lemma 2.5]{bd2014} it follows immediately that
\begin{equation}\label{Upsilonbound2}\Upsilon_{t}(A,B)\ll A^{2t-2}B^{4t-2k+\eps}+A^{2t-1}B^{2t-k+\eps}.\end{equation}
This is only slightly worse than our result, the only difference being that the $(\log B)^{K(t,k)}$ is replaced with $B^{\eps}$. However, the major arcs we will use are smaller than those used in \cite{bd2014} and this will mean that the pointwise bound on the minor arcs will only save an arbitrary power of $\log B$, which is enough to cancel out a $(\log B)^{K(t,k)}$ factor, but not one of size $B^{\eps}$. An alternative approach could have been to enlarge the major arcs so that \eqref{Upsilonbound2} is sufficient. This could probably be done by adapting a result of Liu \cite{liu2010}, but this would likely not lead to any improvements in Theorem \ref{main2}.
\end{remark}
\begin{proof}
We begin by defining
\begin{equation}\label{Falphabetadef}F(\alpha,\beta)=\sum_{0<|a|\le A}\sum_{p\le B}\sum_{q\le B} (\log p)(\log q)e\left(a\left(\alpha p^{k}+\beta q^{k}\right)\right)\end{equation}
where $p,q$ run over primes. We can then rewrite $\Upsilon_{t}(A,B)$ as
\[\Upsilon_{t}(A,B)=\int_{0}^{1}\int_{0}^{1} F(\alpha,\beta)^{2t}d\alpha d\beta=\int_{0}^{1}\int_{0}^{1} \left| F(\alpha,\beta)\right|^{2t}d\alpha d\beta.\]
(The second part is because $F(\alpha,\beta)$ is real.) We can now split $F$ up into pieces $F_{i,j}$, ($i,j\le L=\lfloor 1+\log B/\log 2\rfloor$) where this is the same sum as in $F$, but restricted to $2^{-i}B<p\le 2^{1-i}B$, $2^{-j}B<q\le 2^{1-j}B$. Then
\[F(\alpha,\beta)=\sum_{i=1}^{L}\sum_{j=1}^{L} F_{i,j}(\alpha,\beta).\]
Applying H\"older's inequality then gives:
\[\Upsilon_{t}(A,B)\le L^{4t}\max_{i,j\le L} \int_{0}^{1}\int_{0}^{1}\left| F_{i,j}(\alpha,\beta)\right|^{2t} d\alpha d\beta.\]
Note that $L\ll \log B$ so we have
\[\Upsilon_{t}(A,B)\ll (\log B)^{4t} \max_{\substack{1\le 2X\le B\\1\le 2Y\le B}} \Psi(A,X,Y)\]
where $\Psi(A,X,Y)=\sum\prod_{j}(\log p_{j})(\log q_{j})$ and the sum is over solutions to equation \eqref{doublemain} with
\begin{equation}\label{psiconditions}
0<|a_{j}|\le A,\quad X<p_{j}\le 2X,\quad Y<q_{j}\le 2Y
\end{equation}
As in \cite[Lemma 2.5]{bd2014} we split $\Psi$ into two pieces: $\Psi'$ which consists of the solutions in which $(p_{1}^{k},\dots,p_{2t}^{k})$ and $(q_{1}^{k},\dots,q_{2t}^{k})$ are linearly independent; and $\Psi''$ which consists of the linearly dependent solutions. Clearly:
\[\Psi(A,X,Y)=\Psi'(A,X,Y)+\Psi''(A,X,Y).\]
However, $\Psi''$ will be a lot easier to bound than in the case considered in \cite{bd2014}. We claim that if $(p_{1}^{k},\dots,p_{2t}^{k})$ and $(q_{1}^{k},\dots,q_{2t}^{k})$ are linearly dependent then either they are equal, or $p_{1}=p_{2}=\dots=p_{2t}$ and $q_{1}=q_{2}=\dots=q_{2t}$. In the latter case there are only $O((B/\log B)^{2})$ possibilities for $p_{j},q_{j}$, and $a_{1}+\cdots+a_{2t}=0$ so there are $O(A^{2t-1})$ possibilities for $a_{j}$. Hence these terms contribute $O(A^{2t-1}B^{2})$, which we will see is negligible compared to the solutions with $(p_{1},\dots,p_{2t})=(q_{1},\dots,q_{2t})$. But first we need to prove the claim.

Suppose $(p_{1}^{k},\dots,p_{2t}^{k})=\lambda (q_{1}^{k},\dots,q_{2t}^{k})$, $\lambda\in\bQ$. Then for all $i,j$, $p_{i}q_{j}=p_{j}q_{i}$. Suppose that $p_{i}\ne q_{i}$ for some $i$, then for all $j$, we must have, by unique prime factorisation, $p_{i}=p_{j}$ and $q_{i}=q_{j}$. Hence the claim follows.

Now consider the case $\bp=\bq$. For each $\bp=(p_{1},\dots,p_{2t})$ with $p_{j}\le B$, the set of $\ba$ such that $a_{1}p_{1}^{k}+\cdots+a_{2t}p_{2t}^{k}=0$ is given by the dual lattice $\Lambda^{*}$ of the lattice $\Lambda$ spanned by the vector $(p_{1}^{k},\dots,p_{2t}^{k})$. We may suppose that the $p_{j}$ are not all the same, so that $G(\Lambda)=(p_{1}^{k};\dots;p_{2t}^{k})=1$, and so $\det \Lambda^{*}=\det \Lambda=(p_{1}^{2k}+\cdots+p_{2t}^{2k})^{1/2}$. Hence, the assumption $A\ge B^{2k}$ implies that $A\gg \det \Lambda$, so by Lemma \ref{latticebound}, the number of $|\ba|\le A$ lying in $\Lambda^{*}$ is $O(A^{2t-1}/\det \Lambda)=O(A^{2t-1}/|\bp|^{k})$. Summing over $\bp$ gives
\begin{align*}
\Psi''(A,X,Y)&\ll A^{2t-1}B^{2}+A^{2t-1}\sum_{|\bp| \le B} \frac{1}{|\bp|^{k}}\prod_{j} (\log p_{j})^{2}\\ &\ll A^{2t-1}B^{2}+A^{2t-1}(\log B)^{4t}\sum_{p\le B} \frac{1}{p^{k}} \left( \sum_{p'\le p} 1 \right)^{2t-1}\\ &\ll A^{2t-1}B^{2}+(\log B)^{4t}A^{2t-1} B^{2t-k}.\end{align*}
The assumption $2t\ge k+2$ then implies that this is $O(A^{2t-1}B^{2t-k}(\log B)^{4t})$.

Now for $\Psi'(A,X,Y)$: suppose we have linearly independent $\bp=(p_{1},\dots,p_{2t})$, $\bq=(q_{1},\dots,q_{2t})$ and let $\Delta_{i,j}=p_{i}^{k}q_{j}^{k}-p_{j}^{k}q_{i}^{k}$, $D=\gcd_{1\le i<j\le 2t} \Delta_{i,j}$. The set of $\ba$ such that \eqref{doublemain} is satisfied now forms the dual lattice of the lattice spanned by $(p_{1}^{k},\dots,p_{2t}^{k})$ and $(q_{1}^{k},\dots,q_{2t}^{k})$. This has rank $2t-2$. By \cite[Lemma 2.1 and equation (2.3)]{bd2014}, the determinant is:
\[D^{-1}\left( \sum_{1\le i<j\le 2t} |\Delta_{i,j}|^{2}\right)^{\frac{1}{2}}.\]
The constraints \eqref{psiconditions} imply that $|\Delta_{i,j}|\le 2(4XY)^{k}\le 2B^{2k}$. This means the determinant is $O(A)$ and so by \ref{latticebound}, we get the bound
\[\Psi'(A,X,Y)\ll (\log B)^{4t}A^{2t-2} \sum_{\substack{\bp, \bq \\ (\ast)}} \frac{D}{\max |\Delta_{i,j}|},\]
where ($\ast$) is shorthand for the conditions $X<p_{j}\le 2X$ and $Y<q_{j}\le 2Y$ for all $j$. By symmetry we may suppose $\Delta_{1,2}=\max|\Delta_{i,j}|\ge 1$.

For given $p_{1},p_{2},q_{1},q_{2},D$, let $\Omega(p_{1},p_{2},q_{1},q_{2},D)$ denote the number of possible values for $p_{3},\dots,p_{2t}$, $q_{3},\dots,q_{2t}$ such that $\bp,\bq$ satisfy ($\ast$) together with $\max|\Delta_{i,j}|\le \Delta_{1,2}$ and $D|\Delta_{i,j}$ for all $i,j$. If we sum over the possible values of $p_{1},p_{2},q_{1},q_{2}$ then we get
\[\Psi'(A,X,Y)\ll (\log B)^{4t}A^{2t-2} \sum_{\substack{p_{1},p_{2},q_{1},q_{2}\\ \Delta_{1,2}\ge 1\\ (\ast)}} \sum_{D|\Delta_{1,2}} \frac{D}{\Delta_{1,2}}\Omega(p_{1},p_{2},q_{1},q_{2},D).\]
Now we bound $\Omega$ by discarding most of the conditions, keeping only ($\ast$) and $D|\Delta_{2l-1,2l}$ for $2\le l\le t$. Let $H(X,Y,D)$ be the number of $p_{3},p_{4},q_{3},q_{4}$ satisfying ($\ast$) and $D|\Delta_{3,4}=(p_{3}q_{4})^{k}-(p_{4}q_{3})^{k}$. Then:
\[\Omega(p_{1},p_{2},q_{1},q_{2},D)\le H(X,Y,D)^{t-1}.\]
Lemma \ref{semiprimepower}, together with $\log D\ll \log B$, supplies the bound
\[H(X,Y,D)\le XY\tau(D)^{\lambda(k)}\log_{(3)} B+(XY)^{2}\tau(D)^{\lambda(k)}D^{-1}\log_{(3)} B.\]
So
\begin{align*}
\Psi'(A,X,Y)\ll (\log B)^{4t+\eps}A^{2t-2}\sum_{\substack{p_{1},p_{2},q_{1},q_{2}\\ \Delta_{1,2}\ge 1\\ (\ast)}} \sum_{D|\Delta_{1,2}} &\frac{D}{\Delta_{1,2}}\tau(D)^{(t-1)\lambda(k)}\\&\left((XY)^{t-1}+(XY)^{2t-2}D^{-t+1}\right).
\end{align*}

Now, suppose $k\ge 3$. Then $2t\ge k+2\ge 5$, but since $2t$ is even we in fact have $2t\ge 6$, so $-t+2\le -1$, and
\begin{multline}\label{miscpsi}
\Psi'(A,X,Y)\ll\\ (\log B)^{4t+\eps}A^{2t-2}\sum_{\substack{p_{1},p_{2},q_{1},q_{2}\\ \Delta_{1,2}\ge 1 \\ (\ast)}} \left((XY)^{t-1+\eps}+(XY)^{2t-2}\frac{1}{\Delta_{1,2}}\sum_{D|\Delta_{1,2}}\tau(D)^{(t-1)\lambda(k)}D^{-1}\right).
\end{multline}
By Lemma \ref{divisorsumofdivisor},
\[\sum_{D\mid \Delta_{1,2}}\tau(D)^{(t-1)\lambda(k)}D^{-1}\ll (\log_{(2)}\Delta_{1,2})^{O(1)}\ll (\log B)^{\eps}.\]
Hence:
\[\Psi'(A,X,Y)\ll A^{2t-2}B^{2t+2+\eps}+(\log B)^{4t+\eps}A^{2t-2}(XY)^{2t-2}\Phi(X,Y)\]
where
\[\Phi(X,Y)=\sum_{\substack{p_{1},p_{2},q_{1},q_{2}\\ \Delta_{1,2}\ge 1}} \Delta_{1,2}^{-1}.\]
Writing $x=p_{1}q_{2}$, $y=p_{2}q_{1}$, so that $\Delta_{1,2}=x^{k}-y^{k}\ge 1$. The conditions on $p_{i},q_{i}$ imply $XY\le y<x\le 4XY$, and each $x,y$ in this range corresponds to at most $4$ possible values for $p_{1},p_{2},q_{1},q_{2}$, so
\begin{align*}\Phi(X,Y)&\ll \sum_{XY\le y<x\le 4XY} \frac{1}{x^{k}-y^{k}}\\&\ll \sum_{XY\le y<x\le 4XY} \frac{1}{(x-y)(XY)^{k-1}}\ll (XY)^{2-k}\log B.\end{align*}
The result then follows in the $k\ge 3$ case by noting that $A\ge B^{2k}$ implies $A^{2t-2}B^{2t+2+\eps}\le A^{2t-1}B^{2t-k}$.

Now consider the case $k=2$ and $2t\ge 4$. This is similar but we have $\lambda(2)=1$ and $2-t\le 0$ so $\sum_{D\mid \Delta_{1,2}}\tau(D)^{t-1}D^{2-t}\le \tau(\Delta_{1,2})^{t}$. Instead of \eqref{miscpsi} we have 
\begin{align*}
\Psi'(A,X,Y)\ll (\log B)^{4t+\eps}A^{2t-2}\sum_{\substack{p_{1},p_{2},q_{1},q_{2}\\ \Delta_{1,2}\ge 1 \\ (\ast)}} \tau(\Delta_{1,2})^{t}\left((XY)^{t-1}+(XY)^{2t-2}\Delta_{1,2}^{-1}\right).
\end{align*}
The problem is then reduced to bounding
\[\Phi'(X,Y)=\sum_{\substack{p_{1},p_{2},q_{1},q_{2}\\ \Delta_{1,2}\ge 1}} \tau(\Delta_{1,2})^{t}\Delta_{1,2}^{-1}\ll \sum_{XY\le y<x\le 4XY} \frac{\tau(x^{2}-y^{2})^{t}}{x^{2}-y^{2}}\]
and
\[\Phi''(X,Y)=\sum_{\substack{p_{1},p_{2},q_{1},q_{2}\\ \Delta_{1,2}\ge 1}} \tau(\Delta_{1,2})^{t}\ll \sum_{XY\le y<x\le 4XY} \tau(x^{2}-y^{2})^{t}.\]
We  substitute $u=x-y$, $v=x+y$ and use the inequality $\tau(ab)\le\tau(a)\tau(b)$ as well as the fact that $\sum_{n\le N} \tau(n)^{t}/n \ll (\log N)^{2^{t}}$  (\cite[Lemma 2.4]{hua}) to get
\[\Phi'(X,Y)\ll \sum_{1\le u\le 3XY} \frac{\tau(u)^{t}}{u} \sum_{2XY\le v\le 8XY} \frac{\tau(v)^{t}}{v}\ll (\log B)^{2^{t+1}}.\]
Similarly we have (using \cite[Lemma 2.5]{hua} this time):
\[\Phi''(X,Y)\ll (XY)^{2}(\log B)^{2^{t+1}-2}\]
and
\begin{align*}\Psi'(A,X,Y)&\ll (\log B)^{K(t,2)}A^{2t-2}\left((XY)^{t+1}+(XY)^{2t-2}\right)\\&\le (\log B)^{K(t,2)}A^{2t-2}\left(B^{2t+k}+B^{4t-2k}\right).\end{align*}
To conclude, note that $A\ge B^{2k}$ implies $A^{2t-2}B^{2t+k}\le A^{2t-1}B^{2t-k}$.
\end{proof}
\end{lemma}

\section{The major and minor arcs}\label{arcsection}
\subsection{The major and minor arc decomposition}
The major arcs we will need are
\[\fM(q,r)=\left\{ \alpha :\Vert\alpha-r/q\Vert\le \frac{Q}{AB^{k}}\right\}\cap [0,1]\]
where $Q=(\log B)^{\sigma}$ for some constant $\sigma>0$ to be chosen later and $\Vert \cdot\Vert$ denotes the distance to the nearest integer. When $q\le Q$, $(r;q)=1$, $0\le r<q$, and $B$ is sufficiently large, these will be disjoint intervals of length $2Q/AB^{k}$, except for $\fM(1,0)$ which is a union of two intervals. Then we let
\[\fM=\bigcup_{q\le Q} \bigcup_{0\le r< q} \fM(q,r).\]
The minor arcs are then $\fm=[0,1]\setminus \fM$.
We then let, for $\fA\in \{\fM,\fm\}$:
\begin{equation}\label{rhoAdef}
\rho_{\ba}(B,\fA)=\int_{\fA} g_{\ba}(\alpha)d\alpha,
\end{equation}
and
\begin{equation}\label{rhotildeAdef}
\tilde{\rho}_{\ba}(B,\fA)=\int_{\fA} \tilde{g}_{\ba}(\alpha)d\alpha.
\end{equation}
\subsection{The major arcs}
We wish to approximate $\rho(B,\fM)$ and $\tilde{\rho}(B,\fM)$. The main result will be:
\begin{lemma}\label{majorarcs}
Suppose $s\ge \max\{5,k+1\}$ and $\log A\asymp \log B$. For each $\gamma>0$, and $\sigma>0$ sufficiently large in terms of $\gamma$ we have
\[\sum_{|\ba|\le A}\left| \rho_{\ba}(B,\fM)-\fS_{\ba}J_{\ba}B^{s-k}\right|^{2}\ll \frac{A^{s-2}B^{2s-2k}}{(\log B)^{\gamma}}.\]
Furthermore, if $\psi\ge 3\sigma+\gamma/2$ then
\[\sum_{|\ba|\le A}\left| \tilde{\rho}_{\ba}(B,\fM)-\fS_{\ba}J_{\ba}B^{s-k}\right|^{2}\ll \frac{A^{s-2}B^{2s-2k}}{(\log B)^{\gamma}}.\]
\end{lemma}
To prove this we will need the Siegel-Walfisz theorem. Define
\[\Lambda(n)=\begin{cases} \log p &\text{if }n=p^{k}\text{ for some $p$ and $k>0$}
\\ 0 &\text{otherwise}\end{cases},\]
\[\psi_{q,r}(x)=\sum_{\substack{n\le x\\ n\equiv r \text{ (mod }q\text{)}}} \Lambda(n),\]
and
\[\theta_{q,r}(x)=\sum_{\substack{p\le x\\ p\equiv r \text{ (mod }q\text{)}}} \log p.\]
\begin{thm}[Siegel-Walfisz]\label{siegelwalfisz}
Let $N>0$. Then there exists a constant $c_{N}$ such that for all $x>1$ and all positive integers $q,a$ such that $(a;q)=1$ and $q\le (\log x)^{N}$, we have
\[\psi_{q,a}(x)=\frac{x}{\varphi(q)}+O\left(x\exp(-c_{N}\sqrt{\log x})\right).\]
\begin{proof}
See \cite[Chapter 22]{davenport}.
\end{proof}
\end{thm}
What we will actually need is something which looks more general, but is actually just an easy corollary of Theorem \ref{siegelwalfisz}.
\begin{cor}\label{swcor}
Let $N>0$. Then there exists a constant $c_{N}$ such that for all $1<y\le x$ and all positive integers $q,a$ such that $(a;q)=1$ and $q\le (\log x)^{N}$, we have
\[\theta_{q,a}(y)=\frac{y}{\varphi(q)}+O\left(x\exp(-c_{N}\sqrt{\log x})\right).\]
\begin{proof}
First, it can easily be shown that
\[\theta_{q,a}(y)-\psi_{q,a}(y)\ll y^{1/2}\log y\ll x^{1/2+\eps}.\]
Hence if $q\le (\log y)^{2N}$ then the result follows by applying Siegel-Walfisz with $c'_{N}=c_{2N}$.

So we may suppose $q>(\log y)^{2N}$. In this case, $(\log y)^{2N}<(\log x)^{N}$ which implies
\[y\le \exp\left(\sqrt{\log x}\right)\ll x\exp\left(-c'_{N}\sqrt{\log x}\right).\]
Now simply note that $y/\varphi(q)\le y$ and $\theta_{q,a}(y)\le \theta(y)\ll y$.
\end{proof}
\end{cor}
Now let
\[v(\beta,B)=\int_{0}^{B} e\left(\beta x^{k}\right) dx\]
and
\[v_{1}(\beta,B)=\sum_{1\le n\le B^{k}} \left(n^{1/k}-(n-1)^{1/k}\right)e(\beta n).\]
(Note that $v(\beta)=v(\beta,1)$.)
\begin{lemma}\label{gapprox}
For $0<|a|\le A$ and $\alpha\in \fM(q,r)$ (with $q\le Q= (\log B)^{\sigma}$) we have
\[g(a\alpha)=\varphi(q)^{-1}W(q,ar)v\left(a(\alpha-r/q),B\right)+O\left(B\exp(-c\sqrt{\log B})\right)\]
for some $c$ depending on $\sigma$, and
\[\tilde{g}(a\alpha)=\varphi(q)^{-1}W(q,ar)v\left(a(\alpha-r/q),B\right)+O\left(B(\log B)^{-\psi}\right).\]
\begin{proof}
First note that we may deduce the second formula from the first by the fact that $|g(\alpha)-\tilde{g}(\alpha)|\le B(\log B)^{-\psi}$ for all $\alpha\in\bR$.

Hence we just need to prove the first formula. The proof is similar to \cite[Lemma 7.15]{hua}. We let $\beta\equiv \alpha-r/q$ (mod $1$) with $|\beta|\le 1/2$, and
\[\nu(n)=\begin{cases} e\left(\frac{ar}{q}p^{k}\right)\log p\quad &\text{if }n=p^{k}\text{ for some prime p}\\
0&\text{otherwise}\end{cases},\]
so that
\[g(a\alpha)=\sum_{n\le B^{k}} \nu(n)e(a\beta n).\]
Let $S(X)=\sum_{n\le X} \nu(n)$ and $N=\lfloor B^{k}\rfloor$ so that by partial summation we have
\[g(a\alpha)=S(N)e\left(a\beta (N+1)\right)-\sum_{n\le N}S(n)\left(e\left(a\beta (n+1)\right)-e(a\beta n)\right).\]
By Corollary \ref{swcor} (recall that $q\le (\log B)^{\sigma}$) we have, for all $X\le N$,
\begin{align*}
S(X)&=\sum_{\substack{0\le h<q\\ (h;q)=1}} \theta_{q,h}\left(X^{1/k}\right)e\left(\frac{ar}{q}h^{k}\right)+O\left((\log q)^{2}\right)\\&=\sum_{\substack{0\le h<q\\ (h;q)=1}} \frac{X^{1/k}}{\varphi(q)}e\left(\frac{ar}{q}h^{k}\right)+O\left(B\exp(-c\sqrt{\log B})\right)\\&=X^{1/k}\varphi(q)^{-1}W(q,ar)+O\left(B\exp(-c\sqrt{\log B})\right).
\end{align*}
Note also that $e(a\beta(n+1))-e(a\beta n)\ll A|\beta|$ so
\begin{align*}
g(a\alpha)=&\frac{N^{1/k}}{\varphi(q)}W(q,ar)e\left(a\beta(N+1)\right)\\&-\sum_{n\le N}\frac{n^{1/k}}{\varphi(q)}W(q,ar)\left(e\left(a\beta (n+1)\right)-e(a\beta n)\right)\\&+O\left(\left(1+|\beta|AB^{k}\right)B\exp\left(-c\sqrt{\log B}\right)\right).
\end{align*}
Applying partial summation again in the opposite direction, and noting that $|\beta|AB^{k}\le Q\ll \exp((c/2) \sqrt{\log B})$ we get:
\[g(a\alpha)=\varphi(q)^{-1}W(q,ar)v_{1}(a\beta,B)+O\left(B\exp\left(-\frac{c}{2}\sqrt{\log B}\right)\right).\]

It remains to replace $v_{1}$ with $v$. Since $|W(q,ar)/\varphi(q)|\le 1$, it is enough to show that
\[v(a\beta,B)-v_{1}(a\beta,B)\ll B\exp\left(-\frac{c}{2}\sqrt{\log B}\right).\]
For each $n$ we have
\begin{align*}
\left(n^{1/k}-(n-1)^{1/k}\right)e(a\beta n)&=\int_{n-1}^{n} \frac{1}{k}x^{1/k-1}e(a\beta n)dx\\&=\int_{n-1}^{n} \frac{1}{k}x^{1/k-1}\left(e(a\beta x)+O(A|\beta|)\right)dx\\&=\int_{n-1}^{n} \frac{1}{k}x^{1/k-1}e(a\beta x)dx\\&+O\left(A|\beta|\left(n^{1/k}-(n-1)^{1/k}\right)\right).
\end{align*}
Hence
\begin{align*}
v_{1}(a\beta,B)&=\int_{0}^{N} \frac{1}{k}x^{1/k-1}e(a\beta x)dx+O(AB|\beta|)\\ &=\int_{0}^{B^{k}}\frac{1}{k}x^{1/k-1}e(a\beta x)dx+O\left(B^{1-k}\right)+O(AB|\beta|)\\&=v(a\beta,B)+O\left(QB^{1-k}\right).
\end{align*}
\end{proof}
\end{lemma}
Now we use this to prove Lemma \ref{majorarcs}. Let $\ba\in\bZ^{s}$, $|\ba|\le A$, $R\ge 1$, and define
\[\fS_{\ba}(R)=\sum_{q>R} T_{\ba}(q)\]
and
\[J_{\ba}(R)=\int_{-R}^{R} \prod_{j}v\left(a_{j}\xi\right)d\xi,\]
where we recall that
\[T_{\ba}(q)=\varphi(q)^{-s}\sum_{\substack{r=1\\(r;q)=1}}^{q}\prod_{j}W(q,a_{j}r).\]
\begin{proof}[Proof of Lemma \ref{majorarcs}]
Suppose $\alpha\in \fM(q,r)$ and $\beta=\alpha-r/q$, then by a standard telescoping sum argument (using Lemma \ref{gapprox} and $|g(a\alpha)|\le B$), we get:
\begin{equation}\label{majorarcseq1}g_{\ba}(\alpha)=\varphi(q)^{-s}\prod_{j} W\left(q,a_{j}r\right)v\left(a_{j}\beta,B\right)+O\left(B^{s}\exp\left(-c\sqrt{\log B}\right)\right)\end{equation}
and
\begin{equation}\tilde{g}_{\ba}(\alpha)=\varphi(q)^{-s}\prod_{j} W\left(q,a_{j}r\right)v\left(a_{j}\beta,B\right)+O\left(B^{s}(\log B)^{-\psi}\right).\end{equation}
If we integrate over $\fM(q,r)$ then we get, for $g_{\ba}$:
\begin{equation}\label{majorarcseq2}\varphi(q)^{-s}\prod_{j} W\left(q,a_{j}r\right)\int_{-\frac{Q}{AB^{k}}}^{\frac{Q}{AB^{k}}} v_{\ba}(\beta,B)d\beta+O\left(A^{-1}B^{s-k}\exp\left(-\frac{c}{2}\sqrt{\log B}\right)\right)\end{equation}
and for $\tilde{g}_{\ba}$:
\begin{equation}\varphi(q)^{-s}\prod_{j} W\left(q,a_{j}r\right)\int_{-\frac{Q}{AB^{k}}}^{\frac{Q}{AB^{k}}} v_{\ba}(\beta,B)d\beta+O\left(A^{-1}B^{s-k}(\log B)^{\sigma-\psi}\right).\end{equation}
It can easily be shown, by integrating by substitution, that $v(a\beta,B)=Bv(a\beta B^{k})$. Using this and the substitution $\xi=B^{k}\beta$ gives:
\begin{align}\label{majorarcseq3} \nonumber
\int_{-\frac{Q}{AB^{k}}}^{\frac{Q}{AB^{k}}} v_{\ba}(\beta,B)d\beta=B^{s}\int_{-\frac{Q}{AB^{k}}}^{\frac{Q}{AB^{k}}} v_{\ba}(B^{k}\beta)d\beta&=B^{s-k}\int_{-Q/A}^{Q/A} v_{\ba}(\xi)d\xi\\&=B^{s-k}J_{\ba}\left(Q/A\right).
\end{align}
Let $E_{\ba}=J_{\ba}-J_{\ba}(Q/A)$. Now if we sum over $r,q\le Q$ we get:
\begin{align*}
\rho_{\ba}(B,\fM)&=\fS_{\ba}J_{\ba}B^{s-k}-\fS_{\ba}(Q)J_{\ba}(Q/A)B^{s-k}\\&-\fS_{\ba}E_{\ba}B^{s-k}+O\left(A^{-1}B^{s-k}\exp\left(-\frac{c}{4}\sqrt{\log B}\right)\right)
\end{align*}
and
\begin{align*}
\tilde{\rho}_{\ba}(B,\fM)&=\fS_{\ba}J_{\ba}B^{s-k}-\fS_{\ba}(Q)J_{\ba}(Q/A)B^{s-k}\\&-\fS_{\ba}E_{\ba}B^{s-k}+O\left(A^{-1}B^{s-k}(\log B)^{3\sigma-\psi}\right).
\end{align*}
Hence if we let
\[V_{1}=\sum_{|\ba|\le A} \left|\fS_{\ba}(Q)J_{\ba}(Q/A)B^{s-k}\right|^{2}\]
and
\[V_{2}=\sum_{|\ba|\le A} \left|\fS_{\ba}E_{\ba}B^{s-k}\right|^{2},\]
then
\[\sum_{|\ba|\le A}\left|\rho_{\ba}(B,\fM)-\fS_{\ba}J_{\ba}B^{s-k}\right|^{2}\ll V_{1}+V_{2}+A^{s-2}B^{2s-2k}\exp\left(-\frac{c}{4}\sqrt{\log B}\right)\]
and
\[\sum_{|\ba|\le A}\left|\tilde{\rho}_{\ba}(B,\fM)-\fS_{\ba}J_{\ba}B^{s-k}\right|^{2}\ll V_{1}+V_{2}+A^{s-2}B^{2s-2k}(\log A)^{6\sigma-2\psi}.\]

So the result will follow provided we can bound $V_{1}$ and $V_{2}$. By the identity $v(\beta)\ll \min(1,|\beta|^{-1/k})$ and H\"older's inequality:
\begin{align*}
|J_{\ba}(R)|&\le \left(\prod_{j}\int_{-R}^{R}\left|v\left(a_{j}\beta\right)\right|^{s}d\beta\right)^{1/s}=\left(\prod_{j}|a_{j}|^{-1}\int_{-R|a_{j}|}^{R|a_{j}|}|v(\xi)|^{s}d\xi\right)^{1/s}\\ &\ll \left|a_{1}a_{2}\cdots a_{s}\right|^{-1/s}\int_{-\infty}^{\infty} \min\left(1,|\xi|^{-s/k}\right)d\xi\ll \left|a_{1}a_{2}\cdots a_{s}\right|^{-1/s}
\end{align*}
where the integral converges because $s>k$. This implies
\[V_{1}\ll \sum_{|\ba|\le A}\left|a_{1}a_{2}\cdots a_{s}\right|^{-2/s}\fS_{\ba}(Q)^{2}B^{2s-2k}.\]
Applying a similar argument to $E_{\ba}$:
\begin{align}
\nonumber|E_{\ba}|&\le \left(\prod_{j}\int_{|\beta|>Q/A}\left|v\left(a_{j}\beta\right)\right|^{s}d\beta\right)^{1/s}=\left(\prod_{j}|a_{j}|^{-1}\int_{|\xi|>|a_{j}|Q/A}|v(\xi)|^{s}d\xi\right)^{1/s}\\ \label{majorarcseq4} &\ll\left(\prod_{j}\left|a_{j}\right|^{-1}\int_{|a_{j}|Q/A}^{\infty} \xi^{-s/k}d\xi\right)^{1/s}\ll\left|a_{1}a_{2}\cdots a_{s}\right|^{-1/k}(A/Q)^{(s-k)/k}.
\end{align}
So
\[V_{2}\ll \sum_{|\ba|\le A}\left|a_{1}a_{2}\cdots a_{s}\right|^{-2/k}\fS_{\ba}^{2}A^{2(s-k)/k}B^{2s-2k}Q^{-2(s-k)/k}.\]
Hence if we can show that, for all $s\ge 5$ and $0\le\tau< 1$:
\begin{equation}\label{majorarcsclaim}
\sum_{|\ba|\le A}\fS_{\ba}(R)^{2}\left|a_{1}a_{2}\cdots a_{s}\right|^{-\tau}\ll_{\tau} A^{s(1-\tau)}R^{\eps-(s-4)}
\end{equation}
and
\begin{equation}\label{majorarcsclaim2}
\sum_{|\ba|\le A}\fS_{\ba}(R)^{2}\left|a_{1}a_{2}\cdots a_{s}\right|^{-1}\ll (\log A)^{s}R^{\eps-(s-4)},
\end{equation}
then it will follow that,
\[V_{1}\ll A^{s-2}Q^{\eps-(s-4)}B^{2s-2k}\]
and for $k\ge 3$:
\[V_{2}\ll A^{s-2}Q^{-2(s-k)/k}B^{2s-2k}\]
and for $k=2$:
\[V_{2}\ll A^{s-2}(\log A)^{s}Q^{-2(s-k)/k}B^{2s-2k}.\]
Given that $Q=(\log B)^{\sigma}\asymp (\log A)^{\sigma}$, choosing $\sigma$ and $\psi\ge3\sigma+\gamma/2$ sufficiently large would then complete the proof.

It remains to show \eqref{majorarcsclaim} and \eqref{majorarcsclaim2}. Recall \eqref{Tabound}:
\[T_{\ba}(q)\ll q^{1-s/2+\eps}\prod_{j}(a_{j};q)^{1/2},\]
which implies, for $s\ge 5$
\[\fS_{\ba}(R)^{2}\ll \sum_{q>R}\sum_{r>R}(qr)^{1-s/2+\eps}\prod_{j}\left(qr;a_{j}\right).\]
Hence
\[\sum_{|\ba|\le A}\fS_{\ba}(R)^{2}\left|a_{1}a_{2}\cdots a_{s}\right|^{-\tau}\ll \sum_{q>R}\sum_{r>R}(qr)^{1-s/2+\eps} \sum_{|\ba|\le A}\prod_{j}\left(qr;a_{j}\right)|a_{j}|^{-\tau}.\]
By splitting up according to the possible values of $(qr;a_{j})$ we get that, when $0\le\tau<1$,
\begin{align*}\sum_{|\ba|\le A}\prod_{j}\left(qr;a_{j}\right)|a_{j}|^{-\tau}&=\left(\sum_{0<|a|\le A} (qr;a)|a|^{-\tau}\right)^{s}\le \left(\sum_{d\mid qr}\sum_{1\le |b|\le A/d}d^{1-\tau}|b|^{-\tau}\right)^{s}\\&\ll_{\tau} \left(\sum_{d\mid qr}A^{1-\tau}\right)^{s}\ll(qr)^{\eps}A^{s(1-\tau)}.\end{align*}
When $\tau=1$ it is $O((qr)^{\eps}(\log A)^{s})$ by a similar argument. We are then left with
\[\sum_{q>R}\sum_{r>R}(qr)^{1-s/2+\eps}\ll \left(R^{2-s/2+\eps}\right)^{2},\]
which gives \eqref{majorarcsclaim} and \eqref{majorarcsclaim2}.
\end{proof}
\subsection{The minor arcs}
The goal of this section is to bound $g$ pointwise on the minor arcs.
\begin{lemma}\label{weyltype1}
Let $L=\log B$ and suppose $\log A\asymp \log B$. Then for any $\sigma_{0}>0$, and $\sigma>0$ sufficiently large in terms of $\sigma_{0}$, we have that for all $\alpha\in \fm$:
\[\sum_{0<|a|\le A}|g(a\alpha)|^{2^{k-1}}\ll AB^{2^{k-1}}L^{-\sigma_{0}}\]
and (for all $\psi>0$)
\[\sum_{0<|a|\le A}|\tilde{g}(a\alpha)|^{2^{k-1}}\ll AB^{2^{k-1}}L^{-\sigma_{0}},\]
where the implied constants depend on $\sigma$, $\sigma_{0}$ and $\psi$. (Note that $\fm$ also depends on $\sigma$.)
\end{lemma}
\begin{lemma}\label{weyltype2}
Define:
\[f(\beta)=\sum_{0<n\le B} e\left(\beta n^{k}\right).\]
For each $\sigma_{1}>0$, and for $\sigma>0$ sufficiently large in terms of $\sigma_{1}$ we have that for all $\alpha\in \fm$:
\[\sum_{0<|a|\le A}\left| f(a\alpha)\right|^{2^{k-1}}\ll \frac{AB^{2^{k-1}}}{(\log B)^{\sigma_{1}}}.\]
\begin{proof}
By Dirichlet's approximation theorem there are $r,q\in \bZ$ with
\[1\le q\le \frac{AB^{k}}{Q}=\frac{AB^{k}}{(\log B)^{\sigma}},\]
\[|q\alpha-r|\le \frac{Q}{AB^{k}}.\]
If $q\le Q$ then $\alpha\in \fM$, a contradiction, so we have
\[(\log B)^{\sigma}<q\le \frac{AB^{k}}{(\log B)^{\sigma}}.\]
We let $K=2^{k-1}$ and use Weyl differencing as in the usual proof of Weyl's inequality \cite[Lemma 2.4]{vaughan} to get:
\[\left| f(\alpha)\right|^{K}\ll B^{K-k}\left(B^{k-1}+\sum_{0<|\bh|\le B}\min\{B,\Vert\alpha k!h_{1}\cdots h_{k-1}\Vert^{-1}\}\right),\]
where $\bh=(h_{1},\dots,h_{k-1})$. Summing over $a$ gives
\begin{align*}
\sum_{0<|a|\le A}\left| f(a\alpha)\right|^{K}&\ll AB^{K-1}+B^{K-k}\sum_{0<|a|\le A}\sum_{0<|\bh|\le B}\min\{B,\Vert\alpha k!ah_{1}\cdots h_{k-1}\Vert^{-1}\}\\&\ll AB^{K-1}+B^{K-k}\sum_{n\le k!AB^{k-1}}\tau_{k}(n)\min\{B,\Vert\alpha n\Vert^{-1}\},
\end{align*}
where $\tau_{k}(n)$ is the number of ways of writing $n$ as a product of $k$ positive integers. One can easily show that $\tau_{k}(n)\le \tau(n)^{k}$. We can therefore use \cite[Lemma 6.1]{hua} which states that for $\sigma_{2}\ge 2^{3k}-1$:
\[\sum_{\substack{n\le M\\ (\log M)^{\sigma_{2}}\le \tau(n)^{k}}}\tau(n)^{k}\ll M(\log M)^{-\sigma_{2}}.\]
Applying this with $\sigma_{2}=\max\{\sigma_{1},2^{3k}-1\}$ and taking $\sigma=\sigma_{2}+\sigma_{1}+1$ and noting that $\log (AB^{k-1})\asymp \log B$ gives:
\[\sum_{0<|a|\le A}\left| f(a\alpha)\right|^{K}\ll AB^{K-1}+\frac{AB^{K}}{(\log B)^{\sigma_{2}}}+(\log B)^{\sigma_{2}}B^{K-k}\sum_{n\le k!AB^{k-1}}\min\{B,\Vert\alpha n\Vert^{-1}\}.\]
The result then follows from \cite[Lemma 2.2]{vaughan}, which implies that:
\[\sum_{x\le X} \min\{Y,\Vert\alpha x\Vert^{-1}\}\ll XY\left(\frac{1}{q}+\frac{1}{Y}+\frac{1}{X}+\frac{q}{XY}\right)\log(2Xq),\]
for all $X,Y\ge 1$.
\end{proof}
\end{lemma}
We are now ready to prove Lemma \ref{weyltype1}. The idea will be to replace $g$ with $f$ and then use Lemma \ref{weyltype2}.
\begin{proof}[Proof of Lemma \ref{weyltype1}]
To simplify notation, let $m=2^{k-2}$, $\bp=(p_{1},\dots,p_{m})$ and $\bq=(q_{1},\dots,q_{m})$ where $p_{j}$ and $q_{j}$ are primes. Further, let 
\[L(\bp,\bq)=(\log p_{1})\cdots (\log p_{m})(\log q_{1})\cdots (\log q_{m})\]
and
\[G(\alpha,x_{1},\dots,x_{m},y_{1},\dots,y_{m})=\sum_{0<|a|\le A} e\left(a\alpha\left(x_{1}^{k}+\cdots+x_{m}^{k}-y_{1}^{k}-\cdots-y_{m}^{k}\right)\right).\]
Then by expanding the sum and using Cauchy's inequality:
\begin{align}
\nonumber\sum_{0<|a|\le A} \left|g(a\alpha)\right|&^{2^{k-1}}\\
\nonumber&=\sum_{0<|a|\le A}\sum_{|\bp|\le B}\sum_{|\bq|\le B}L(\bp,\bq)e\left(a\alpha\left(p_{1}^{k}+\cdots+p_{m}^{k}-q_{1}^{k}-\cdots-q_{m}^{k}\right)\right)\\
\nonumber&=\sum_{|\bp|\le B}\sum_{|\bq|\le B}L(\bp,\bq)G(\alpha,\bp,\bq)\\
\label{weyltype1eq1}&\le B^{2^{k-2}}\left(\sum_{|\bp|\le B}\sum_{|\bq|\le B}L(\bp,\bq)^{2}|G(\alpha,\bp,\bq)|^{2}\right)^{1/2}.
\end{align}
Now that the summand is nonnegative we can replace the sum over primes with the same sum over nonnegative integers $\bx=(x_{1},\dots,x_{m})$, $\by=(y_{1},\dots,y_{m})$:
\begin{align*}
&\sum_{|\bp|\le B}\sum_{|\bq|\le B}L(\bp,\bq)^{2}|G(\alpha,\bp,\bq)|^{2}\\&\le (\log B)^{2^{k}}\sum_{0<|\bx|\le B}\sum_{0<|\by|\le B}|G(\alpha,\bx,\by)|^{2}\\&=(\log B)^{2^{k}}\sum_{0<|\bx|\le B}\sum_{0<|\by|\le B}\sum_{0<|a_{1}|\le A}\sum_{0<|a_{2}|\le A}\\& \qquad \qquad \qquad e\left(\left(a_{1}-a_{2}\right)\alpha\left(x_{1}^{k}+\cdots+x_{m}^{k}-y_{1}^{k}-\cdots-y_{m}^{k}\right)\right)\\&=(\log B)^{2^{k}}\sum_{0<|a_{1}|\le A}\sum_{0<|a_{2}|\le A}\left|f\left((a_{1}-a_{2})\alpha\right)\right|^{2^{k-1}}\\&\ll (\log B)^{2^{k}}AB^{2^{k-1}}+(\log B)^{2^{k}}A\sum_{0<|a|\le 2A} \left|f(a\alpha)\right|^{2^{k-1}},
\end{align*}
where in the last line we substituted $a=a_{1}-a_{2}$ and the $AB^{2^{k-1}}$ term comes from the $a=0$ summands. We now use Lemma \ref{weyltype2} with $\sigma_{1}=2\sigma_{0}+2^{k}$ and plug the result back into \eqref{weyltype1eq1} to complete the proof for $g$. To get the part about $\tilde{g}$, we have
\[\tilde{g}(a\alpha)=g(a\alpha,B)-g\left(a\alpha,B(\log B)^{-\psi}\right),\]
so
\[\sum_{0<|a|\le A}|\tilde{g}(a\alpha)|^{2^{k-1}}\ll \sum_{0<|a|\le A} \left(|g(a\alpha,B)|^{2^{k-1}}+\left|g\left(a\alpha,B(\log B)^{-\psi}\right)\right|^{2^{k-1}}\right).\]
We now apply the first part of the lemma with $B(\log B)^{-\psi}$ in place of $B$, noting that $\log (B(\log B)^{-\psi})\asymp \log A$.
\end{proof}
We are now ready to prove the final minor arc estimate. Recall the definitions \eqref{rhoAdef} and \eqref{rhotildeAdef} of $\rho_{\ba}$ and $\tilde{\rho}_{\ba}$ respectively.
\begin{lemma}\label{minorarcs}
Let $s=2t+u$ with $t\in \bN$, $2t\ge k+2$ and $u=1$ or $2$. Then for each $\gamma>0$ and $\sigma>0$ sufficiently large in terms of $\gamma$, we have for all $2\le B^{2k}\le A\le B^{2t-k}$:
\[\sum_{0<|\ba|\le A} \left| \rho_{\ba}(B,\fm)\right|^{2}\ll \frac{A^{s-2}B^{2s-2k}}{(\log B)^{\gamma}}\]
and
\[\sum_{0<|\ba|\le A} \left| \tilde{\rho}_{\ba}(B,\fm)\right|^{2}\ll \frac{A^{s-2}B^{2s-2k}}{(\log B)^{\gamma}}.\]
\begin{proof}
As in \cite[Lemma 4.4]{bd2014}:
\[\sum_{|\ba|\le A}\left| \rho_{\ba}(B,\fm)\right|^{2}=\sum_{|\ba|\le A}\int_{\fm} g_{\ba}(\alpha) d\alpha\int_{\fm} g_{\ba}(-\beta)d\beta=\int_{\fm}\int_{\fm} F(\alpha,-\beta)^{s}d\alpha d\beta,\]
where $F$ is as defined in \eqref{Falphabetadef}. By Lemma \ref{meanval}, and the assumption that $A\le B^{2t-k}$:
\begin{equation}\label{meanvalbound}\int_{0}^{1}\int_{0}^{1} F(\alpha,-\beta)^{2t}d\alpha d\beta\ll (\log B)^{K(t,k)}A^{2t-2}B^{4t-2k}.\end{equation}

Cauchy's inequality implies that, for all $\alpha,\beta\in\bR$,
\[|F(\alpha,-\beta)|\le \left(\sum_{0<|a|\le A}|g(a\alpha)|^{2}\right)^{1/2}\left(\sum_{0<|a|\le A}|g(a\beta)|^{2}\right)^{1/2}.\]
Then by H\"older's inequality and Lemma \ref{weyltype1} we have,
\begin{align}\label{pointwisebound}
\sup_{\alpha,\beta\in \fm}|F(\alpha,-\beta)|&=\sup_{\alpha\in\fm}\sum_{0<|a|\le A} |g(a\alpha)|^{2}\\ &\ll \sup_{\alpha\in\fm}A^{1-2^{2-k}}\left(\sum_{0<|a|\le A} |g(a\alpha)|^{2^{k-1}}\right)^{2^{2-k}}\ll \frac{AB^{2}}{(\log B)^{\sigma_{0}2^{2-k}}},
\end{align}
where $\sigma_{0}>0$ is arbitrary and $\sigma>0$ depends on $\sigma_{0}$. Hence choosing $\sigma_{0}$ such that \[\sigma_{0}2^{2-k}-K(t,k)\ge\gamma\] gives the first part.

The second part is similar, but instead of $F$, we have
\[\tilde{F}(\alpha,\beta)=\sum_{1\le|a|\le A}\sum_{B(\log B)^{-\psi}<p\le B}\sum_{B(\log B)^{-\psi}<q\le B} (\log p)(\log q)e\left(a\left(\alpha p^{k}+\beta q^{k}\right)\right).\]
And by considering the underlying Diophantine equation we get:
\[\int_{0}^{1}\int_{0}^{1}\tilde{F}(\alpha,-\beta)^{2t}d\alpha d\beta\le \int_{0}^{1}\int_{0}^{1}F(\alpha,-\beta)^{2t}d\alpha d\beta.\]
We can then apply \eqref{meanvalbound}. We also get, from the second part of Lemma \ref{weyltype1} and a similar argument to \eqref{pointwisebound}:
\[\sup_{\alpha,\beta\in\fm}|F(\alpha,-\beta)|\ll \frac{AB^{2}}{(\log B)^{\sigma_{0}2^{2-k}}}.\]
The result then follows in the same way as the first part, and with the same choice of $\sigma_{0}$.
\end{proof}
\end{lemma}
\begin{proof}[Proof of Theorem \ref{main2}]
Theorem \ref{main2} now follows immediately from Lemma \ref{majorarcs} and Lemma \ref{minorarcs}.
\end{proof}

Now we have proven Theorems \ref{main2} and \ref{main3}, we can use them to deduce Theorem \ref{main}.
\begin{proof}[Proof of Theorem \ref{main}]For $2\le B^{2k}\le A\le B^{\hat{s}-k}$, Theorem \ref{main2} gives us the following bound:
\[\#\left\{ |\ba|\le A: \left|\rho_{\ba}(B)-\fS_{\ba}J_{\ba}B^{s-k}\right|>\frac{B^{s-k}}{A(\log B)^{\gamma/3}}\right\}\ll \frac{A^{s}}{(\log A)^{\gamma/3}}.\]
(Noting that $\log A\asymp \log B$.) If $\ba$ is not in the above set, then
\[\rho_{\ba}(B)\ge \fS_{\ba}J_{\ba}B^{s-k}-\frac{B^{s-k}}{A(\log B)^{\gamma/3}}.\]
Suppose $\ba$ is also not in the sets described in \eqref{main3eq1} and \eqref{main3eq2}. If \eqref{maineq} has solutions in $\bR^{+}$ then $J_{\ba}\gg A^{-1}(\log A)^{-\eta}$ and if it has solutions in $\bZ_{p}^{\times}$ for all $p$ then $\fS_{\ba}\ge (\log A)^{-\eta}$. Hence, provided $\eta<\gamma/6$, we have that $\rho_{\ba}(B)>0$ for $A$, $B$ sufficiently large. Because $\gamma$ and $\eta$ are arbitrary, and $\delta\asymp \eta$, we may choose $\eta$ to be large enough that $\delta\ge\lambda$, $\eta\ge \lambda$ and then choose $\gamma>6\eta$ so that $\gamma/6>\eta$ and $\gamma/3>\lambda$. This gives us a solution in the range $x_{j}\le B$ for all $|\ba|\le A$ except for a set of size at most $O_{\lambda}(A^{s}/(\log A)^{\lambda})$.

Suppose $A/2< |\ba|\le A$ and take $B=A^{1/(\hat{s}-k)}$ (which is the largest that we can, given $2\le B^{2k}\le A\le B^{\hat{s}-k}$). Then \eqref{maineq} has a solution in primes $x_{j}\ll |\ba|^{1/(\hat{s}-k)}$ for all $A/2<|\ba|\le A$, $\ba\in\mathcal{C}'(k,s)$ except for a set of size
\[O_{\lambda}\left(\frac{A^{s}}{(\log A)^{\lambda}}\right).\]
Now we apply this to the ranges $A/2^{n+1}<|\ba|\le A/2^{n}$ for each $n\ge 0$ and sum over $n$ (we may suppose $2^{n}\le A/2$ because otherwise there are at most $O(1)$ choices of $\ba$). We get that the number of exceptional $\ba$ is at most
\begin{align*}&\ll_{\lambda} 1+\sum_{2^{n}\le A/2} \frac{A^{s}}{2^{sn}(\log A/2^{sn})^{\lambda}}\\ &\ll_{\lambda} 1+\sum_{2^{n}\le A^{1/2}} \frac{A^{s}}{2^{sn}(\log A)^{\lambda}}+\sum_{A^{1/2}<2^{n}\le A/2} (A/2^{n})^{s}\\ &\ll_{\lambda} \frac{A^{s}}{(\log A)^{\lambda}}.\end{align*}

For the second part, for each $\lambda>0$, we choose $\eta$, $\delta$ and $\gamma$ as above, and then for $\psi$ sufficiently large in terms of $\gamma$, the second part of Theorem \ref{main2} will hold. Hence, if we follow the same argument as above, we get the desired result.
\end{proof}

\section{The density of locally soluble equations}\label{localdensitysection}
In this section we prove Theorem \ref{main4}. We need to understand when \eqref{maineq} has a solution in $\bZ_{p}^{\times}$ for all $p$, and the following lemma allows us to do this when $p$ is large.
\begin{lemma}\label{chinonzero}
Let $k\ge 2$ and $s\ge 3$. For $r$ a positive integer, let
\[p_{0}(s,r)=\min\left\{p\text{ prime}:(p-1)^{s-1}p^{-s/2}>r^{s-1}\right\}.\]
Then for any prime $p$ such that $p\nmid k$ and $p\ge p_{0}(s,(k;p-1))$, and any $\ba=(a_{1},\dots,a_{s})$ such that at least $3$ of the $a_{j}$ are not divisible by $p$, we have $\chi_{p}>0$. In particular, if $p\ge p_{0}(s,k)$ then $p\nmid k$ so the above conclusion holds for all sufficiently large $p$.

Furthermore, we have
\[p_{0}(s,r)\le (2r)^{2\frac{s-1}{s-2}}+1.\]
\begin{proof}
Recall the definitions \eqref{xidef} of $\xi(p)$ and \eqref{lambdadef} of $\lambda(p)$. By Lemma \ref{singserlemma1}, it is enough to show that there is a solution in $(\bZ/p^{\xi(p)+\lambda(p)}\bZ)^{\times}$. Since at least $3$ of the $a_{j}$ are not divisible by $p$, we have $\lambda(p)=0$ and the assumptions on $p$ imply that $p\nmid k$, so $\xi(p)=1$. Hence it is enough to show that there is a solution in $(\bZ/p\bZ)^{\times}$. Since the $a_{j}$ with $p\mid a_{j}$ are all 0 (mod $p$) we may throw these away and suppose without loss of generality that $p\nmid a_{j}$ for all $j$ and that $s\ge 3$.

We apply the (mod $p$) version of the circle method. Recall that $M_{\ba}(p)$ is the number of solutions to
\[a_{1}x_{1}^{k}+\cdots+a_{s}x_{s}^{k}\equiv 0\:\text{(mod $p$)}\]
with $0<x_{j}<p$. We have
\begin{align}\nonumber M_{\ba}(p)&=\sum_{0<x_{1}<p}\dots\sum_{0<x_{s}<p}\frac{1}{p}\sum_{m=0}^{p-1}e\left(\frac{m}{p}(a_{1}x_{1}^{k}+\cdots+a_{s}x_{s}^{k})\right)\\\nonumber &=\frac{1}{p}\sum_{m=0}^{p-1}\left(\prod_{j=1}^{s}W(p,a_{j}m)\right)\\&=\frac{(p-1)^{s}}{p}+\frac{1}{p}\sum_{m=1}^{p-1}\left(\prod_{j=1}^{s}W(p,a_{j}m)\right).\end{align}
Since $a_{j}$ is invertible (mod $p$), we can show by reordering the sum that
\[\sum_{m=1}^{p-1} |W(p,a_{j}m)|^{s}=\sum_{m=1}^{p-1} |W(p,m)|^{s}.\]
Hence, by H\"older's inequality:
\begin{align}\nonumber \left|M_{\ba}(p)-\frac{(p-1)^{s}}{p}\right|&\le \frac{1}{p}\sum_{m=1}^{p-1}\prod_{j=1}^{s}|W(p,a_{j}m)|\\\nonumber &\le \frac{1}{p}\prod_{j=1}^{s}\left(\sum_{m=1}^{p-1} |W(p,a_{j}m)|^{s}\right)^{1/s}\\&=\frac{1}{p}\sum_{m=1}^{p-1} |W(p,m)|^{s}.\end{align}

Now, we know from \cite[Lemma 8.5]{hua} that when $(m;q)=1$, we have that $W(q,m)\ll q^{1/2+\eps}$, but if we go into the proof then we find that this may be somewhat improved when $q$ is a prime. We have that, for any $p\nmid \lambda$,
\[W(p,m)=\sum_{x=1}^{p-1} e_{p}(mx^{k})=\sum_{x=1}^{p-1} e_{p}(m(\lambda x)^{k})=W(p,\lambda^{k}m).\]
Since $\lambda^{k}$ runs through $(p-1)/(k;p-1)$ congruence classes (mod $p$), we get, for $p\nmid m$,
\[\frac{p-1}{(k;p-1)}|W(p,m)|^{2}\le \sum_{r=1}^{p-1}|W(p,r)|^{2}\le \sum_{r=0}^{p-1}|W(p,r)|^{2}.\]
The sum $\sum_{r=0}^{p-1}|W(p,r)|^{2}$ is equal to $p$ times the number of solutions to the equation
\[x_{1}^{k}\equiv x_{2}^{k}\:\text{(mod $p$)}\]
subject to $0<x_{j}<p$. For each $x_{1}$, there are exactly $(k;p-1)$ choices for $x_{2}$, giving $(p-1)(k;p-1)$ solutions in total. Hence,
\[\sum_{r=0}^{p-1}|W(p,r)|^{2}=p(p-1)(k;p-1),\]
and for $p\nmid m$,
\[|W(p,m)|\le p^{1/2}(k;p-1).\]
When $s\ge 3$,
\begin{align*}\frac{1}{p}\sum_{m=1}^{p-1} |W(p,m)|^{s}&\le \frac{1}{p}\left(\max_{0<m<p} |W(p,m)|\right)^{s-2}\sum_{m=1}^{p-1}|W(p,r)|^{2}\\ &\le p^{(s-2)/2}(k;p-1)^{s-1}(p-1).\end{align*}
Therefore $M_{\ba,n}(p)>0$ provided that
\[\frac{(p-1)^{s}}{p}>p^{(s-2)/2}(k;p-1)^{s-1}(p-1),\]
which is equivalent to
\[(p-1)^{s-1}p^{-s/2}>(k;p-1)^{s-1}.\]
We can easily check, for example by differentiating, that $(x-1)^{s-1}x^{-s/2}$ is an increasing function in $x>1$, so the above inequality is equivalent to $p\ge p_{0}(s,(k;p-1))$. So $M_{\ba}(p)>0$ (and hence $\chi_{p}>0$) provided that $p\ge p_{0}(s,(k;p-1))$, as required. It just remains to show the upper bound for $p_{0}(s,r)$. For this, note that if $p\ge 2$ then $p-1\ge p/2$. So
\[(p-1)^{s-1}p^{-s/2}\ge 2^{1-s}p^{(s-2)/2},\]
and $2^{1-s}p^{(s-2)/2}>r^{s-1}$ if and only if $p>(2r)^{2(s-1)/(s-2)}$.
\end{proof}
\end{lemma}
We say that a set of numbers $a_{1},\dots,a_{s}$ is $r$-coprime if for any prime $p$, if $p\mid (a_{1};\dots;a_{s})$ then $p\mid r$. We say that they are $k$-wise $r$-coprime if every subset of size $k$ is $r$-coprime. Note that the usual idea of coprime is the same as $1$-coprime in our definition. We also say that a set is $k$-wise $r$-coprime to $m$ if the greatest common divisor of any subset of size $k$ is $r$-coprime to $m$.

In \cite{h2013}, Hu found the probability that a random $s$-tuple is $k$-wise coprime, which generalised a result of T\'oth \cite{t2002} in which it was found the probability that an $s$-tuple is pairwise ($2$-wise) coprime. We generalise the result further to include congruence constraints on the tuples. Our proof will closely follow \cite{h2013}, although unfortunately it will be even more notationally complicated.

Given positive integers $s,k,r,n$, and $\bu=(u_{1},\dots,u_{k-1})$, $\bq=(q_{1},\dots,q_{s})$ and $\bb=(b_{1},\dots,b_{s})$ tuples of positive integers, define $Q_{s,k,\bb}^{(r,\bq,\bu)}(n)$ to be the number of $(a_{1},\dots,a_{s})$ which are $k$-wise $r$-coprime and $i$-wise $r$-coprime to $u_{i}$ for each $1\le i\le \min\{s,k-1\}$, such that $1\le a_{j}\le n$ and $a_{j}\equiv b_{j}$ (mod $q_{j}$) for all $j$. We are only interested in the case $\bu=(1,\dots,1)$, where $Q_{s,k,\bb}^{(r,\bq,\bu)}(n)$ is just the count of $k$-wise $r$-coprime tuples congruent to $\bb$, but we will use induction, and for the inductive step we need to consider general $\bu$.

Define:
\[A_{s,k,r}=\prod_{p\,\nmid \,r}\sum_{m=0}^{k-1}\binom{s}{m}\left(1-p^{-1}\right)^{s-m}p^{-m},\]
(where we say that $\binom{s}{m}=0$ for $m>s$),
\[A_{s,k,r,\bq}=\left(\prod_{j=1}^{s}q_{j}^{-1}\right)A_{s,k,r},\]
\[f_{s,k,r,i}(u)=\prod_{\substack{p\mid u\\p\,\nmid\, r}}\left(1-\frac{\sum_{m=i}^{k-1}\binom{s}{m}(p-1)^{k-1-m}}{\sum_{m=0}^{k-1}\binom{s}{m}(p-1)^{k-1-m}}\right),\]
\[\delta(s,k)=\max\left\{\binom{s-1}{i},i=1,2,\dots,k-1\right\},\]
and $\theta(u)$ to be the number of squarefree divisors of $u$. 
\begin{lemma}\label{kcoprimedense}
Let $s,k$ be fixed positive integers. Then for all positive integers $n,r$ and tuples of positive integers $\bu=(u_{1},\dots,u_{k-1})$, $\bq=(q_{1},\dots,q_{s})$, $\bb=(b_{1},\dots,b_{s})$ such that $q_{i}\mid r$ for all $i$, and $r,u_{1},\dots,u_{k-1}$ are pairwise coprime, we have:
\begin{equation}\label{kcoprimedenseeq1}Q_{s,k,\bb}^{(r,\bq,\bu)}(n)=A_{s,k,r,\bq}\left(\prod_{i=1}^{k-1}f_{s,k,r,i}(u_{i})\right)n^{s}+O\left(\theta(u_{1})n^{s-1}(\log n)^{\delta(s,k)}\right),\end{equation}
where the implied constant depends only on $s$ and $k$. Hence, the probability that a tuple of positive integers $(a_{1},\dots,a_{s})$ is $k$-wise $r$-coprime and $a_{j}\equiv b_{j}$ (mod $q_{j}$) is $A_{s,k,r,\bq}$.
\begin{remark}
 Setting $r=1$, $\bq=\bb=(1,\dots,1)$ gives the main result of \cite{h2013}. Note that $k$ has a different meaning here to elsewhere in the paper. This is an unfortunate clash of notation but we have chosen it to be consistent with \cite{h2013}.
\end{remark}
\begin{proof}
As in \cite{h2013} we let $(a;b]$ be the product of all $p^{e}$ with $p$ prime, $p\mid a$ and $p^{e}\Vert b$. Then $b/(a;b]$ is the number obtained from $b$ by removing from its prime factorisation all the primes which divide $a$, and so $a,b$ are $r$-coprime if and only if $a,b/(r;b]$ are coprime.

As in \cite{h2013}, let
\[\alpha_{s,i}(d)=d^{i}\prod_{p\mid d}\sum_{m=0}^{i}\binom{s}{m}\left(1-p^{-1}\right)^{i-m}p^{-m},\]
where in this formula, $\binom{s}{i}=0$ for $i>s$ and $0^{0}=1$.

If we let $f_{s,k,i}=f_{s,k,1,i}$, then
\[f_{s,k,r,i}(u_{i})=f_{s,k,r,i}\left(\frac{u_{i}}{(r;u_{i}]}\right)=f_{s,k,i}\left(\frac{u_{i}}{(r;u_{i}]}\right).\]
It follows from this, \cite[Lemma 4]{h2013} and the multiplicativity of $f_{s,k,r,i}$, that:
\begin{equation}\label{kcoprimedenseeq6}\frac{f_{s,k,r,i}(u_{i})}{f_{s,k,r,i+1}(u_{i})}=\sum_{\substack{d\mid u_{i}\\(d;r)=1}} \frac{\mu(d)\binom{s}{i}^{\omega(d)}}{\alpha_{s,i}(d)}\end{equation}
for $1\le i\le k-2$, and
\begin{equation}\label{kcoprimedenseeq7}f_{s,k,r,k-1}(u_{k-1})=\sum_{\substack{d\mid u_{k-1}\\(d;r)=1}}\frac{\mu(d)\binom{s}{k-1}^{\omega(d)}}{\alpha_{s,k-1}(d)}.\end{equation}
As in \cite{h2013} define $j\ast \bu$ by
\[j\ast \bu=\left(u_{1}(j;u_{2}),\frac{u_{2}(j;u_{3})}{(j;u_{2}]},\dots,\frac{u_{k-2}(j;u_{k-1})}{(j;u_{k-2}]},\frac{ju_{k-1}}{(\prod_{i=2}^{k-1}[j;u_{i}))(j;u_{k-1}]}\right).\]
Letting $\bb=(b_{1},\dots,b_{s+1})$, $\bb'=(b_{1},\dots,b_{s})$, $\bq=(q_{1},\dots,q_{s+1})$, $\bq'=(q_{1},\dots,q_{s})$, we have the recurrence relation:
\begin{equation}\label{kcoprimedenserecurrence}Q_{s+1,k,\bb}^{(r,\bq,\bu)}(n)=\sum_{\substack{j=1\\(j;u_{1}/(r;u_{1}])=1\\j\equiv b_{s+1}\text{(mod $q_{s+1}$)}}}^{n}Q_{s,k,\bb'}^{(r,\bq',j\ast\bu)}(n).\end{equation}
The proof of \eqref{kcoprimedenserecurrence} is very similar to that of \cite[Lemma 3]{h2013}, so we omit it.

We now prove the lemma by induction on $s$. For $s=1$ we have
\[Q_{1,k,b}^{(r,q,u)}(n)=\sum_{\substack{j=1\\(j;u/(r;u])=1\\j\equiv b\text{(mod $q$)}}}^{n} 1=\sum_{\substack{d\mid u\\(d;r)=1}}\mu(d)\sum_{\substack{e\le n/d\\de\equiv b\text{(mod $q$)}}} 1,\]
where we used the fact that $d\mid u/(r;u]$ if and only if $d\mid u$ and $(d;r)=1$. For each $d$ in the sum, $(d;r)=1$ and $q\mid r$, so $d$ is invertible (mod $q$) and $de\equiv b$ becomes $e\equiv d^{-1}b$. It follows that
\begin{equation}\label{kcoprimedenseeq4}\sum_{\substack{e\le n/d\\de\equiv b\text{(mod $q$)}}} 1=\frac{n}{dq}+O(1)\end{equation}
and
\begin{align}\label{kcoprimedenseeq3}Q_{1,k,b}^{(r,q,u)}(n)=\sum_{\substack{d\mid u\\(d;r)=1}}\mu(d)\left(\frac{n}{dq}+O(1)\right)&=\frac{n}{q}\sum_{\substack{d\mid u\\(d;r)=1}}\frac{\mu(d)}{d}+O\left(\sum_{d\mid u}\mu(d)^{2}\right)\nonumber\\\nonumber&=\frac{n}{q}\frac{\varphi(u/(r;u]))}{u/(r;u]}+O(\theta(u))\\&=\frac{n}{q}A_{1,k,r}f_{1,k,r,1}(u)+O(\theta(u)).\end{align}
Here we used the identities $A_{1,k,r}=1$ and $f_{1,k,r,1}(u)=\varphi(u/(r;u])/(u/(r;u])$, which can easily be shown.

Now suppose the lemma holds for $s$. Let $\bb'$, $\bq'$ be as above and, to simplify notation, let $b=b_{s+1}$ and $q=q_{s+1}$. We may also assume that $(u_{i};r)=1$ for all $i$ by replacing $u_{i}$ with $u_{i}/(r;u_{i}]$ and noting that both sides of equation \eqref{kcoprimedenseeq1} remain unchanged (except that the error term becomes smaller). The inductive hypothesis and \eqref{kcoprimedenserecurrence} then give:
\begin{align*}&Q_{s+1,k,\bb}^{(r,\bq,\bu)}(n)=\sum_{\substack{j=1\\(j;u_{1})=1\\j\equiv b\text{(mod $q$)}}}^{n}Q_{s,k,\bb'}^{(r,\bq',j\ast\bu)}(n)\\ &=\sum_{\substack{j=1\\(j;u_{1})=1\\j\equiv b\text{(mod $q$)}}}^{n}A_{s,k,r,\bq'}\left(\prod_{i=1}^{k-2}f_{s,k,r,i}\left(\frac{u_{i}(j;u_{i+1})}{(j;u_{i}]}\right)\right)\\&\times f_{s,k,r,k-1}\left(\frac{ju_{k-1}}{(\prod_{i=2}^{k-1}[j;u_{i}))(j;u_{k-1}]}\right)n^{s}+O\left(\sum_{j=1}^{n}\theta(u_{1}(j;u_{2}))n^{s-1}(\log n)^{\delta(s,k)}\right)\end{align*}
By multiplicativity of $f_{s,k,r,i}$ and the fact that $f_{s,k,r,i}((j;u])=f_{s,k,r,i}((j;u))$ (which can be seen from the definition), this is equal to
\begin{align}\label{kcoprimedenseeq2}Q_{s+1,k,\bb}^{(r,\bq,\bu)}(n)&=A_{s,k,r,\bq'}\left(\prod_{i=1}^{k-1}f_{s,k,r,i}(u_{i})\right)n^{s} \sum_{\substack{j=1\\(j;u_{1})=1\\j\equiv b\text{(mod $q$)}}}^{n}\left(\prod_{i=1}^{k-2}\frac{f_{s,k,r,i}((j;u_{i+1}))}{f_{s,k,r,i+1}((j;u_{i+1}))}\right)\nonumber\\&\times f_{s,k,r,k-1}\left(\frac{j}{\prod_{i=2}^{k-1}[j;u_{i})}\right)+O\left(\theta(u_{1})n^{s-1}(\log n)^{\delta(s,k)}\sum_{j=1}^{n}\theta(j)\right).\end{align}
By a similar argument to that on \cite[page 1268]{h2013}, using \eqref{kcoprimedenseeq6} and \eqref{kcoprimedenseeq7} we get:
\begin{align}\sum_{\substack{j=1\\(j;u_{1})=1\\j\equiv b\text{(mod $q$)}}}^{n}\left(\prod_{i=1}^{k-2}\frac{f_{s,k,r,i}((j;u_{i+1}))}{f_{s,k,r,i+1}((j;u_{i+1}))}\right)f_{s,k,r,k-1}\left(\frac{j}{\prod_{i=2}^{k-1}[j;u_{i})}\right)\nonumber\\=\sum_{\substack{d_{1}d_{2}\cdots d_{k-1}\le n\\d_{i}\mid u_{i+1},i=1,\dots,k-2\\(d_{k-1};u_{i})=1,i=1,2,\dots,k-1\\(d_{k-1};r)=1}}\left(\prod_{i=1}^{k-1}\frac{\mu(d_{i})\binom{s}{i}^{\omega(d_{i})}}{\alpha_{s,i}(d_{i})}\right)\sum_{\substack{e\le \frac{n}{d_{1}d_{2}\cdots d_{k-1}}\\(e;u_{1})=1\\d_{1}\cdots d_{k-1}e\equiv b\text{(mod $q$)}}}1.\end{align}
By \eqref{kcoprimedenseeq3}, we have:
\[\sum_{\substack{e\le \frac{n}{d_{1}d_{2}\cdots d_{k-1}}\\(e;u_{1})=1\\d_{1}\cdots d_{k-1}e\equiv b\text{(mod $q$)}}}1=\frac{n}{qd_{1}d_{2}\cdots d_{k-1}}\frac{\varphi(u_{1})}{u_{1}}+O(\theta(u_{1})).\]
Here we have used the fact that $(u_{1};r)=1$ (so $(r;u_{1}]=1$). If we apply this and use the fact that $\alpha_{s,i}(d)\ge d$, and then use the fact that $\mu(d_{i})$, $\binom{s}{i}^{\omega(d_{i})}$ and $\alpha_{s,i}(d_{i})$ are all multiplicative, we get:
\begin{align}\label{kcoprimedenseeq5}
\sum_{\substack{j=1\\(j;u_{1})=1\\j\equiv b\text{(mod $q$)}}}^{n}&\left(\prod_{i=1}^{k-2}\frac{f_{s,k,r,i}((j;u_{i+1}))}{f_{s,k,r,i+1}((j;u_{i+1}))}\right)f_{s,k,r,k-1}\left(\frac{j}{\prod_{i=2}^{k-1}[j;u_{i})}\right)\nonumber\\=&\frac{\varphi(u_{1})}{qu_{1}}n\sum_{\substack{d_{1}d_{2}\cdots d_{k-1}=d\le n\\d_{i}\mid u_{i+1},i=1,2,\dots,k-2\\(d_{k-1};u_{i})=1,i=1,2,\dots,k-1\\(d_{k-1};r)=1}}\prod_{i=1}^{k-1}\frac{\mu(d_{i})\binom{s}{i}^{\omega(d_{i})}}{d_{i}\alpha_{s,i}(d_{i})}\nonumber\\&+O\left(\theta(u_{1})\sum_{d\le n}\frac{\delta(s+1,k)^{\omega(d)}}{d}\right).
\end{align}
Then, letting
\[x_{i,p}=1-\frac{\binom{s}{i}}{p\sum_{m=0}^{i}\binom{s}{m}(p-1)^{i-m}},\]
we see that the right hand side of \eqref{kcoprimedenseeq5} is equal to
\begin{align}\label{kcoprimedenseeq8}q^{-1}n\left(\prod_{i=0}^{k-2}\prod_{p\mid u_{i+1}}x_{i,p}x_{k-1,p}^{-1}\right)\prod_{p\,\nmid \,r} x_{k-1,p}+&O\left(n\sum_{d>n}\frac{\delta(s+1,k)^{\omega(d)}}{d^{2}}\right)\nonumber\\+&O\left(\theta(u_{1})\sum_{d\le n}\frac{\delta(s+1,k)^{\omega(d)}}{d}\right).\end{align}
(We have used the fact that $\varphi(u_{1})/u_{1})=\prod_{p\mid u_{1}}x_{0,p}$.) We now put \eqref{kcoprimedenseeq8} into equation \eqref{kcoprimedenseeq2}. For the main term we use the fact that
\[x_{i,p}=\frac{\sum_{m=0}^{i}\binom{s+1}{m}(p-1)^{i-m}}{\sum_{m=0}^{i}\binom{s}{m}(p-1)^{i-m}}\left(1-\frac{1}{p}\right),\]
which follows from a a simple computation, and then we deal with the error terms by a similar argument to that in \cite{h2013}.
\end{proof}
\end{lemma}

\begin{proof}[Proof of Theorem \ref{main4}]
Let $P\ge p_{0}(s,k)$, where $p_{0}$ is as in Lemma \ref{chinonzero},
\[R(A)=\#\{\ba\in \mathcal{C}'(k,s):|\ba|\le A\}=\#\{|\ba|\le A:\fS_{\ba}J_{\ba}>0\}\]
and
\[R'(A)=\#\{|\ba|\le A:\fS_{\ba}J_{\ba}>0,(a_{1};\dots;a_{s})=1\}.\]
Also, let $C'_{p}(k,s)$ be the set of $\ba\in C_{p}(k,s)$ such that $p\nmid a_{j}$ for at least one $j$ and define
\[\delta'_{p}=\lim_{A\rightarrow\infty}\frac{\#\{|\ba|\le A: \ba\in C'_{p}(k,s)\}}{(2A)^{s}},\]
\begin{align}\tilde{\delta}_{p}&=\sum_{m=0}^{s-3}\binom{s}{m}(1-p^{-1})^{s-m}p^{-m}\nonumber\\&=1-\frac{1}{p^{s}}-s\left(1-\frac{1}{p}\right)\frac{1}{p^{s-1}}-\frac{s(s-1)}{2}\left(1-\frac{1}{p}\right)^{2}\frac{1}{p^{s-2}}\end{align}
and
\[Q=\prod_{p\le P}p^{\xi(p)},\]
where $\xi(p)$ is as in Lemma \ref{Tazero}. It is not immediately obvious that $\delta'_{p}$ is well-defined, but Lemma \ref{singserlemma1} tells us that $\ba\in C'_{p}(k,s)$ if and only if $p\nmid a_{j}$ for some $j$ and there is a solution to \eqref{maineq} in $(\bZ/p^{\xi(p)}\bZ)^{\times}$. Hence, whether $\ba\in C'_{p}(k,s)$ depends only on the congruence class of $\ba$ (mod $p^{\xi(p)})$. It then follows quite easily that the limit exists. It also follows from the above argument that as long as $C'_{p}(k,s)\ne \emptyset$ then $\delta'_{p}\ge p^{-\xi(p)s}$. We can set $a_{1}=\dots=a_{s-1}=1$, $a_{s}=1-s$ and then $(1,\dots,1)$ will be a solution in $\bZ_{p}^{\times}$ and $\ba\in C'_{p}(k,s)$. Hence, $\delta'_{p}>0$ for all $p$.

Our aim now is to show that for all $A\ge Q$:
\begin{multline*}\delta_{\infty}\left(\prod_{p\le P}\delta'_{p}\right)\left(\prod_{p>P}\tilde{\delta}_{p}\right)+O\left(Q^{s}A^{-1}(\log A)^{\delta(s,s-2)}\right)\le\\ \frac{R'(A)}{(2A)^{s}}\le \delta_{\infty}\prod_{p\le P}\delta'_{p}+O(QA^{-1}).\end{multline*}
For the upper bound, suppose that $(a_{1};\dots;a_{s})=1$ and $\fS_{\ba}J_{\ba}>0$. Then $\ba\in C'_{p}(k,s)$ for every $p$ because $\chi_{p}>0$ implies $\ba\in C_{p}(k,s)$ and coprimality implies $p\nmid a_{j}$ for at least one $j$. As discussed above, $\ba\in C'_{p}(k,s)$ if and only if $p\nmid a_{j}$ for some $j$ and \eqref{maineq} has a solution in $(\bZ/p^{\xi(p)}\bZ)^{\times}$. It follows that there exists a set of tuples $\bb$ with $0\le b_{j}<p^{\xi(p)}$ such that $\ba\in C'_{p}(k,s)$ if and only if for one of these $\bb$, $a_{j}\equiv b_{j}$ (mod $p^{\xi(p)}$) for all $j$. By the Chinese remainder theorem, we have a set $\mathcal{B}$ of tuples $\bb$ with $0\le b_{j}<Q$ such that $\ba\in C'_{p}(k,s)$ for all $p\le P$ if and only if for some $\bb\in \mathcal{B}$, $a_{j}\equiv b_{j}$ (mod $Q$) for all $j$. The condition $J_{\ba}>0$ is equivalent to the $a_{j}$ not all having the same sign. Given $\bb\in \mathcal{B}$, the number of $|\ba|\le A$ with $a_{j}\equiv b_{j}$ (mod $Q$) and with $a_{j}$ not all having the same sign is, assuming $A\ge Q$,
\[(2^{s}-2)\left(A/Q+O(1)\right)^{s}=\delta_{\infty}\frac{(2A)^{s}}{Q^{s}}+O(A^{s-1}Q^{1-s}).\]
Summing over $\mathcal{B}$ and noting that $|\mathcal{B}|=Q^{s}\prod_{p\le P}\delta'_{p}$ then gives the upper bound.

Now for the lower bound, we apply Lemma \ref{kcoprimedense}. Note that for $p>P\ge p_{0}(s,k)$ we have by Lemma \ref{chinonzero} that $\ba\in C'_{p}(k,s)$ provided at least $3$ of the $a_{j}$ are not divisible by $p$. This holds for all $p>P$ if and only if $a_{1},\dots,a_{s}$ are $(s-2)$-wise $Q$-coprime (in the notation of Lemma \ref{kcoprimedense}). If we let $\bb\in \mathcal{B}$ and assume that $a_{j}\equiv b_{j}$ (mod $Q$) and $a_{1},\dots,a_{s}$ are $(s-2)$-wise $Q$-coprime then $\fS_{\ba}>0$. Furthermore, since for all $p\le P$, at least one of the $b_{j}$ is nonzero (mod $p$), we have $p\nmid a_{j}$ for some $j$ and so $(a_{1},\dots,a_{s})=1$. Hence if we also assume that the $a_{j}$ do not all have the same sign then $\ba\in \mathcal{C}'(k,s)$. We can then estimate the number of such $\ba$ using Lemma \ref{kcoprimedense}, with $\bu=(1,\dots,1)$, $r=Q$ and $\bq=(Q,\dots,Q)$. There is, however, a slight issue in that we assumed $a_{j}>0$ for that lemma, but this is easily remedied by considering separately each possible value for $(\mathrm{sign}(a_{1}),\dots,\mathrm{sign}(a_{s}))$ and then applying the lemma to $(|a_{1}|,\dots,|a_{s}|)$. This gives:
\[R'(A)\ge (2^{s}-2)|\mathcal{B}|\frac{A_{s,s-2,Q}}{Q^{s}}A^{s}+O\left(Q^{s}A^{s-1}(\log A)^{\delta(s,s-2)}\right).\]
We have $|\mathcal{B}|/Q^{s}=\prod_{p\le P}\delta'_{p}$ and by looking at the definition, we find that $A_{s,s-2,Q}=\prod_{p>P}\tilde{\delta}_{p}$. Hence we get the lower bound.

Now we establish a formula for $\delta'_{p}$ which holds for any $p$ with the properties that $p>2$, $p\nmid k$ and for all $\ba$ if at least $3$ of the $a_{j}$ are not divisible by $p$ then $\ba\in C_{p}(k,s)$. Recall that this assumption holds for all $p>P$ by Lemma \ref{chinonzero} and that for such $p$, $\xi(p)=1$. It suffices to count the number of $0\le a_{j}<p$ not all zero such that \eqref{maineq} has a solution in $(\bZ/p\bZ)^{\times}$. We separate the $\ba$ according to how many of the $a_{j}$ are $0$. The number of $\ba$ with exactly $i$ of them $0$ is
\[\binom{s}{i}(p-1)^{s-i}.\]
If at least $3$ of the $a_{j}$ are nonzero then there is definitely a solution by Lemma \ref{chinonzero}. The single case where all the $a_{j}$ are $0$ is not counted. In the case where exactly $1$ $a_{j}$ is nonzero, there is never a solution because when we reduce (mod $p$) we get $a_{j}x_{j}^{k}\equiv 0$ and the only solution to this is $x_{j}\equiv 0$. There are $s(p-1)$ of these $\ba$. Now when exactly two are nonzero, say $a_{i},a_{j}$, we have
\[a_{i}x_{i}^{k}+a_{j}x_{j}^{k}\equiv 0\:\text{(mod $p$)},\]
which rearranges to
\[-\frac{a_{i}}{a_{j}}\equiv \left(\frac{x_{j}}{x_{i}}\right)^{k}\:\text{(mod $p$)}\]
and so a solution exists if and only if $-a_{i}/a_{j}$ is a $k$th power residue (mod $p$). There are $(p-1)/(p-1;k)$ of these by \cite[Lemma 8.4]{hua}. Hence, for there to be no solutions, there are $(p-1)$ possible choices for $a_{i}$, and for each of these, there are $(p-1)-(p-1)/(p-1;k)$ choices for $a_{j}$. Including the $s(s-1)/2$ choices for $i,j$ gives a total of
\[\frac{s(s-1)}{2}(p-1)^{2}\left(1-(p-1;k)^{-1}\right)\]
tuples $\ba$ where there is no solution and in which exactly $2$ of the $a_{j}$ are nonzero. Putting all this together, we get that
\begin{align}\label{pre_deltaformula}\delta'_{p}&=p^{-s}\left(p^{s}-1-s(p-1)-\frac{s(s-1)}{2}(p-1)^{2}\left(1-(p-1;k)^{-1}\right)\right)\nonumber\\&=1-\frac{1}{p^{s}}-s\left(1-\frac{1}{p}\right)\frac{1}{p^{s-1}}-\frac{s(s-1)}{2}\left(1-\frac{1}{p}\right)^{2}\left(1-(p-1;k)^{-1}\right)\frac{1}{p^{s-2}}.\end{align}
In particular, $\delta'_{p}=1-O(p^{-(s-2)})$ and when $s\ge 4$,
\[\left|\delta_{\infty}\prod_{p\le P}\delta'_{p}-\delta_{\infty}\prod_{p}\delta'_{p}\right|\ll \sum_{p>P}\frac{1}{p^{s-2}}\ll P^{-(s-3)}.\]
We also have $\tilde{\delta}_{p}=1-O(p^{-(s-2)})$ and so similarly
\[\left|\delta_{\infty}\left(\prod_{p\le P}\delta'_{p}\right)\left(\prod_{p>P}\tilde{\delta}_{p}\right)-\delta_{\infty}\prod_{p}\delta'_{p}\right|\ll P^{-(s-3)}.\]
We therefore have
\[\frac{R'(A)}{(2A)^{s}}=\delta_{\infty}\prod_{p}\delta'_{p}+O\left(P^{-(s-3)}\right)+O\left(Q^{s}A^{-1}(\log A)^{\delta(s,s-2)}\right).\]

It remains to choose $P$ in a way which minimises the error terms. Since $\xi(p)=1$ for all but finitely many $p$, we have,
\[Q\ll \prod_{p\le P} p=\exp\left(\sum_{p\le P}\log p\right)=\exp\left(P+o(P))\right)\ll \exp((1+\eps)P),\]
where for the second equality we used the prime number theorem. Let $c<1/s$ be a small constant and $P=c\log A$, then
\[Q^{s}A^{-1}(\log A)^{\delta(s,s-2)}\ll A^{cs(1+\eps)-1}\ll (\log A)^{-(s-3)}\]
and
\[P^{-(s-3)}\ll (\log A)^{-(s-3)}.\]
Also, $A>Q$ for sufficiently large $A$, so
\[R'(A)=\left(\delta_{\infty}\prod_{p}\delta'_{p}\right)(2A)^{s}+O\left(\frac{A^{s}}{(\log A)^{-(s-3)}}\right).\]

All that remains is to turn this into a similar formula for $R(A)$. Note that $\ba\in \mathcal{C}'(k,s)$ if and only if $n\ba\in \mathcal{C}'(k,s)$ for every nonzero integer $n$, so the number of $\ba\in \mathcal{C}'(k,s)$ with $|\ba|\le A$ and $(a_{1};\dots;a_{s})=d$ is equal to $R'(A/d)$, and hence
\begin{multline}\label{RintermsofR'}R(A)=\sum_{d=1}^{\infty} R'(A/d)=\left(\delta_{\infty}\prod_{p}\delta'_{p}\right)(2A)^{s}\sum_{d=1}^{\infty}d^{-s}+O\left(\sum_{d\le A^{1/2}}\frac{A^{s}}{d^{s}(\log A/d)^{s-3}}\right)\\+O\left(\sum_{d>A^{1/2}}\frac{A^{s}}{d^{s}}\right).\end{multline}
We can easily show that
\[\sum_{d\le A^{1/2}}\frac{A^{s}}{d^{s}(\log A/d)^{s-3}}+\sum_{d>A^{1/2}}\frac{A^{s}}{d^{s}}\ll \frac{A^{s}}{(\log A)^{s-3}}.\]
Then $\sum_{d=1}^{\infty} d^{-s}=\zeta(s)=\prod_{p}(1-p^{-s})^{-1}$, so \eqref{density_limit} will follow provided that
\[\delta_{p}=\delta'_{p}\left(1-\frac{1}{p^{s}}\right)^{-1}.\]
This can be shown in a similar way to \eqref{RintermsofR'} by considering, for each $n$, the $\ba\in C_{p}(k,s)$ with $p^{n}\Vert (a_{1};\dots;a_{s})$ and then summing over $n$. Finally, we also get \eqref{deltaformula} for any $p$ such that \eqref{pre_deltaformula} holds.
\end{proof}
\begin{example}
As an example, we shall calculate $\fD_{4,2}$. We have $\delta_{\infty}=1-2^{1-4}=7/8$. From the proof of Theorem \ref{main4}, we know that when $p\ge p_{0}(4,2)$, $\delta_{p}$ is given by \eqref{deltaformula}. When $s=4$, $k=2$, this is
\[\delta_{p}=1-\left(1-\frac{1}{p^{4}}\right)^{-1}\left(4\left(1-\frac{1}{p}\right)\frac{1}{p^{3}}+6\left(1-\frac{1}{p}\right)^{2}\frac{1}{p^{2}}\right),\]
where we have used the fact that $(2;p-1)=2$ for all $p>2$. We can easily compute $p_{0}(4,2)=11$. For $p<11$ we can use the fact, also from the proof of Theorem \ref{main4}, that $\delta_{p}=\delta'_{p}(1-p^{-4})^{-1}$.

To find $\delta'_{2}$ we have $\xi(2)=3$ so we need to count the number of equations $a_{1}x_{1}^{2}+\cdots+a_{4}x_{4}^{2}=0$ which are soluble (mod $8$) and such that $a_{j}$ are not all even. By a computer search, we find that there are $448$ of these.  Hence $\delta'_{2}=448/8^{4}=7/64$ and $\delta_{2}=\delta'_{2}(15/16)^{-1}=7/60$. For all larger $p$, $\xi(p)=1$. We calculate $\delta'_{3}=26/3^{4}$, $\delta'_{5}=496/5^{4}$, $\delta'_{7}=2268/7^{4}$. Therefore, $\delta_{3}=13/40$, $\delta_{5}=31/39$, $\delta_{7}=189/200$. Hence
\[\fD_{4,2}=\frac{31899}{1280000}\prod_{p\ge 11}\left(1-\left(1-\frac{1}{p^{4}}\right)^{-1}\left(4\left(1-\frac{1}{p}\right)\frac{1}{p^{3}}+6\left(1-\frac{1}{p}\right)^{2}\frac{1}{p^{2}}\right)\right).\]
By evaluating the product up to $p<100$ we find that this is approximately $0.023$.
\end{example}

\section{The density of equations with a prime solution}\label{primesolublesection}
In this section we will prove Theorem \ref{main5}. By \eqref{deltapwhenpbad}, the existence of primes $p$ with $(p-1)\mid k$ means that the density of locally soluble equations will never tend to 1, so we need a different approach. We will consider equations of the form
\begin{equation}\label{maineqinhom}a_{1}x_{1}^{k}+\cdots+a_{s}x_{s}^{k}=n.\end{equation}
By fixing some of the variables, we will turn \eqref{maineq} into something of the form \eqref{maineqinhom} and then we hope to control $n$ so that for each $p$ with $(p-1)\mid k$, there is a solution in $\bZ_{p}^{\times}$.

Define $\chi_{p}(\ba,n)$ analogously to $\chi_{p}(\ba)$ but with \eqref{maineqinhom} in place of \eqref{maineq}. Lemma \ref{chinonzero} can be generalised to show that if at least 3 of the $a_{i}$ are not divisible by $p$ then for all $p\ge p_{0}(s,k)$, $\chi_{p}(\ba,n)>0$. However, for proving Theorem \ref{main5}, we would like to ensure that $\chi_{p}(\ba,n)>0$ for as many $p$ as possible. Assuming that sufficiently many of the $a_{i}$ are not divisible by $p$, we will show that $\chi_{p}(\ba,n)>0$ for all $p$ except those with $(p-1)\mid k$.
\begin{lemma}\label{chinonzero2}
Suppose $p$ is a prime such that $(p-1)\nmid k$. Also let $a_{1},\dots,a_{s},n$ be such that $p\nmid a_{j}$ for at least $3k$ of the $j$s. Then $\chi_{p}(\ba,n)>0$.
\begin{proof}
Without loss of generality suppose that $p\nmid a_{j}$ for $1\le j\le t$. The proof is similar to that of \cite[Lemma 8.8]{hua}, and we will use some lemmas from there. The first is \cite[Lemma 8.4]{hua}, which says that $x^{k}$ takes exactly $(p-1)/(k;p-1)$ distinct values (mod $p$) as $x$ runs over $1\le x\le p-1$. Another is the Cauchy-Davenport lemma \cite[Lemma 8.7]{hua}, which says that if $S_{1},S_{2}\subseteq \bZ/p^{n}\bZ$ such that $x\not\equiv y$ (mod $p$) for all distinct $x,y \in S_{2}$ then
\[\#(S_{1}+S_{2})\ge \min\{\#S_{1}+\#S_{2}-1,p^{n}\}.\]
Applying this inductively to the sets $\{a_{j}x^{k}: 1\le x\le p-1\}$ for $1\le j\le t$ we get that
\begin{multline*}\#\left\{a_{1}x_{1}^{k}+\cdots+a_{t}x_{t}^{k} \text{(mod $p^{\nu(p)+1}$)}: x_{j}\in (\bZ/p^{\nu(p)+1}\bZ)^{\times}\right\}\ge\\ \min\left\{t\frac{p-1}{(p-1;k)}-(t-1),p^{\nu(p)+1}\right\}.\end{multline*}
Hence, if we can show that $t(p-1)/(p-1;k)-(t-1)\ge p^{\nu(p)+1}$ then there will be a solution to $a_{1}x_{1}^{k}+\cdots+a_{s}x_{s}^{k}\equiv n$ (mod $p^{\nu(p)+1}$) with $p\nmid x_{j}$ (just let $x_{j}=1$ for $j>t$). By a similar argument to that given in the proof of Lemma \ref{singserlemma1}, this implies that $\chi_{p}>0$ (note that $(p-1)\nmid k$ so $p$ must be odd and $\xi(p)=\nu(p)+1$, and that if $p^{m}\Vert (a_{1};\dots;a_{s})$ and \eqref{maineqinhom} has solutions in $(\bZ/p^{m+1}\bZ)^{\times}$ then we must also have $p^{m}\mid n$).

It therefore just remains to show that 
\begin{equation}\label{chinonzero2_eq1}t(p-1)/(p-1;k)-(t-1)\ge p^{\nu(p)+1}.\end{equation}
Let, $k_{0}$, $\nu(p)$ be such that $k=p^{\nu(p)}k_{0}$ and $p\nmid k_{0}$. In the proof of \cite[Lemma 8.8]{hua}, it is shown that \eqref{chinonzero2_eq1} holds provided $t\ge 2k$, $(p-1)\nmid k_{0}$ and $k_{0}\nmid (p-1)$. They also show that it holds when $t\ge 2k$, $(p-1)\nmid k_{0}$ and $p-1=mk_{0}$ with $m>2$. All that remains is the case $k=p^{\nu(p)}(p-1)/2$. In this case $(p-1;k_{0})=(p-1)/2$ so if $t\ge 3k$ then
\[t\frac{p-1}{(p-1;k_{0})}-(t-1)=t+1> 3k=p^{\nu(p)}\frac{3(p-1)}{2}\ge p^{\nu(p)+1}.\]
\end{proof}
\end{lemma}
\begin{lemma}\label{inhom}
For each $k\ge 1$, there exists some $s_{0}(k)$ such that the following holds. Let $s\ge s_{0}(k)$ and $\ba=(a_{1},\dots,a_{s})\in \bZ^{s}$, $n\in \bZ$ be such that:
\begin{enumerate}
\item The numbers $a_{j}$ are all nonzero and do not all have the same sign.
\item Equation \eqref{maineqinhom} has a solution in $\bZ_{p}^{\times}$ (or equivalently $\chi_{p}(\ba,n)>0$) for every prime $p$.
\end{enumerate}
Then \eqref{maineqinhom} has a solution in the primes.
\begin{proof}
The proof is a fairly standard application of the circle method and so we will only provide a sketch. In addition to $\chi(\ba,n)$ defined earlier, we also define $\rho_{\ba,n}$, $\fS_{\ba,n}$, $J_{\ba,n}$ and $T_{\ba,n}(q)$ in the obvious way. The assumption that $a_{j}$ do not all have the same sign implies that $J_{\ba,n}=J_{\ba}$ and so in particular this is nonzero. Most of the lemmas from Section \ref{singsersection} generalise easily, and in particular $\fS_{\ba,n}>0$ provided that \eqref{maineqinhom} has a solution in $\bZ_{p}^{\times}$ for all $p$. This holds by the second assumption.

We take $s_{0}(k)=2^{k}+1$ (this can be improved upon, but this will not be necessary for our purposes). We wish to establish the following asymptotic formula:
\begin{equation}\label{inhomasymptotic}\rho_{\ba,n}(X)=J_{\ba}\fS_{\ba,n}X^{s-k}+o\left(X^{s-k}\right).\end{equation}
(Recall that $J_{\ba}=J_{\ba,n}$). Here, $\ba$ and $n$ are considered fixed, so that the error term may depend on them. By the discussion above, it suffices to prove \eqref{inhomasymptotic}.

We let $Q=(\log X)^{\sigma}$ for some $\sigma>0$ and take the following major arcs:
\[\fM(q,r)=\left\{\alpha:\Vert \alpha-r/q\Vert \ll Q/X^{k}\right\}\cap [0,1],\]
which are simpler than those considered in the proof of Theorem \ref{main2} because now $\ba$ is fixed. We of course define the minor arcs to be $\fm=[0,1]\setminus\fM$. Then
\[\rho_{\ba,n}(X)=\int_{0}^{1} g_{\ba}(\alpha,X)e(-n\alpha)d\alpha.\]
For the minor arcs we have:
\[\int_{\fm}  g_{\ba}(\alpha,X)e(-n\alpha)d\alpha\le \left(\prod_{j=2^{k}+1}^{s}\sup_{\alpha\in \fm} |g(a_{j}\alpha)|\right)\int_{0}^{1}\prod_{j=1}^{2^{k}} |g(a_{j}\alpha)|d\alpha.\]
Now by \cite[Theorem 10]{hua} we can show that for all $\sigma_{0}>0$, there exists a choice of $\sigma>0$ such that, for all $1\le j\le s$:
\[\sup_{\alpha\in \fm} |g(a_{j}\alpha)|\ll \frac{X}{(\log X)^{\sigma_{0}}}.\]
By H\"older's inequality followed by a substitution of $\beta=a_{j}\alpha$ we have
\[\int_{0}^{1}\prod_{j=1}^{2^{k}} |g(a_{j}\alpha)|d\alpha\le \int_{0}^{1}|g(\beta)|^{2^{k}}d\beta\le (\log X)^{2^{k}}\int_{0}^{1}|f(\beta)|^{2^{k}}d\beta,\]
where $f$ is defined as in Lemma \ref{weyltype2}. We now bound this integral by $O(X^{2^{k}-k})$ using Hua's lemma (\cite[Theorem 4]{hua}). Combining all this, we get
\[\int_{\fm} g_{\ba}(\alpha)e(-n\alpha)d\alpha=o\left(X^{s-k}\right).\]
For the major arcs, we use Lemma \ref{gapprox} with $A=\max_{j} |a_{j}|$, $B=X$ and take the integral over $\fM(q,r)$ (note that, unlike the Theorem \ref{main2}, Lemma \ref{gapprox} does not require the assumption $A\ge B^{2k}$). This gives, instead of \eqref{majorarcseq2},
\begin{align*}\varphi(q)^{-s}\left(\prod_{j}W(q,a_{j}r)\right)e\left(-\frac{rn}{q}\right)\int_{-QX^{-k}}^{QX^{-k}} v_{\ba}(\beta,X)e(-\beta n)d\beta\\ +O\left(X^{s-k}\exp\left(-c\sqrt{\log X}\right)\right).\end{align*}
Then, proceeding as in \eqref{majorarcseq3},
\begin{align*}\int_{-QX^{-k}}^{QX^{-k}}v_{\ba}(\beta,X)e(-\beta n)d\beta&=X^{s-k}\int_{-Q}^{Q}v_{\ba}(\xi)e\left(-X^{-k}\xi n\right)d\xi\\&=X^{s-k}\left(J_{\ba}(Q)+O\left(Q^{2}X^{-k}n\right)\right).\end{align*}
Since the coefficients do not all have the same sign, we may let $X\rightarrow \infty$ while keeping $n$ fixed, and this means we can ignore the $O(Q^{2}X^{-k}n)$ term. We can then show that $\fS_{\ba,n}(Q)\ll \sum_{q>Q}q^{1-s/2+\eps}$ (which follows from Lemma \ref{singconv}) and $J_{\ba}-J_{\ba}(Q)\ll Q^{-(s-k)/k}$ (which follows from the same argument as \eqref{majorarcseq4}), to finish the major arc analysis.
\end{proof}
\end{lemma}
\begin{proof}[Proof of Theorem \ref{main5}]
Fix $k\ge 2$ and let
\[d(s)=\liminf_{A\rightarrow \infty}\frac{\#\{|\ba|\le A: \text{\eqref{maineq} has a solution in the primes} \}}{(2A)^{s}}.\]

The main obstacles to the proof are those primes $p$ with $(p-1)\mid k$, because in this case, no matter how large $s$ is, we still can not guarantee that $\chi_{p}(\ba)>0$. We list these primes as $p_{1},p_{2},\dots,p_{r}$. Let $t=3k+p_{1}^{\xi(p_{1})}+\cdots+p_{r}^{\xi(p_{r})}$ and suppose $s>t$. Suppose $\ba=(a_{1},\dots,a_{s})$ are $(s-t+1)$-wise coprime, at least $t+1$ of the $a_{j}$ are positive and at least $t+1$ are negative. We claim that, under these assumptions, when $s$ is sufficiently large, $a_{1}x_{1}^{k}+\cdots+a_{s}x_{s}^{k}=0$ has a solution in the primes. Assuming the claim, we will then have that
\[d(s)\ge 1-d_{1}(s)-2d_{2}(s),\]
where
\[d_{1}(s)=\limsup_{A\rightarrow \infty}\frac{\#\{|\ba|\le A: \text{$(a_{1},\dots,a_{s})$ are not $(s-t+1)$-wise coprime}\}}{(2A)^{s}}\]
and
\[d_{2}(s)=\limsup_{A\rightarrow \infty}\frac{\#\{|\ba|\le A: \text{at most $t$ of the $a_{j}$ are positive}\}}{(2A)^{s}}.\]
It will then suffice to prove that $d_{j}(s)\ll s^{\lambda(k)}2^{-s}$ for $j\in \{1,2\}$ and some $\lambda(k)$.

To prove the claim, we first note that  a set is $(s-t+1)$-wise coprime if and only if every prime divides at most $s-t$ of its elements. In other words, for every prime $p$, there are at least $t$ of the $a_{j}$ which are not divisible by $p$. 

Suppose first that $k>2$. If we have some integers $b_{1},\dots,b_{p^{l}}$ which are all not divisible by $p$, then by applying the Cauchy-Davenport lemma \cite[Lemma 8.7]{hua} to the sets $\{0,b_{1}\},\dots,\{0,b_{p^{l}}\}$, we get that for any $m\in \bZ$, there exists a subset of the $b_{j}$ whose sum is congruent to $m$ (mod $p^{l}$). At least $t\ge p_{1}^{\xi(p_{1})}$ of the $a_{j}$ are not divisible by $p_{1}$ and so applying this to these $a_{j}$ and reordering if necessary, we have some $u_{1}\le p_{1}^{\xi(p_{1})}$ such that $a_{s-u_{1}+1}+\cdots+a_{s} \equiv a_{1}+\cdots+a_{s}$ (mod $p_{1}^{\xi(p)}$) and hence
\[a_{1}+\cdots+a_{s-u_{1}}\equiv 0\: \left(\text{mod $p_{j}^{\xi(p_{j})}$}\right).\] Similarly, there are at least $t-u_{1}\ge p_{2}^{\xi(p_{2})}$ of the $a_{1},\dots,a_{s-u_{1}}$ which are not divisble by $p_{2}$, and by reordering $a_{1},\dots,a_{s-u_{1}}$ if necessary, we have some $u_{2}$ with $u_{2}-u_{1}\le p_{2}^{\xi(p_{2})}$ such that \[a_{1}+\cdots+a_{s-u_{2}}+a_{s-u_{1}+1}+\cdots+a_{s}\equiv 0\: \left(\text{mod $p_{j}^{\xi(p_{j})}$}\right).\] By applying this argument inductively, we get that after possibly reordering the $a_{1},\dots,a_{s}$, there exist $u_{1},\dots,u_{r}$ with $u_{j}-u_{j-1}\le p_{j}^{\xi(p_{j})}$ for $1\le j\le r$ (where $u_{0}$ is taken to be $0$), such that
\begin{equation}\label{main5proofeq1}a_{1}+\cdots+a_{s-u_{j}}+a_{s-u_{j-1}+1}+\cdots+a_{s}\equiv 0\: \left(\text{mod $p_{j}^{\xi(p_{j})}$}\right)\end{equation}
for $1\le j\le r$. Let $\tilde{s}=s-u_{r}$. For $1\le i\le r$ and $s-u_{i}+1\le j\le s-u_{i-1}$, let $x_{j}=p_{i}$, and let
\[n=-(a_{\tilde{s}+1}x_{\tilde{s}+1}^{k}+\cdots+a_{s}x_{s}^{k}).\]
Now, consider the congruence
\[a_{1}x_{1}^{k}+\cdots+a_{\tilde{s}}x_{\tilde{s}}^{k}\equiv n\: \left(\text{mod $p_{i}^{\xi(p_{i})}$}\right)\]
and note that for each $p$ there are at least $t-u_{r}\ge 3k$ of the $a_{1},\dots,a_{\tilde{s}}$ not divisible by $p$. We claim that this congruence has a solution in integers $x_{j}$ with $p_{i}\nmid x_{j}$ for all $j\le \tilde{s}$, and hence $\chi_{p_{i}}(\tilde{\ba},n)>0$ for all $1\le i\le r$, where $\tilde{\ba}=(a_{1},\dots,a_{\tilde{s}})$.

To see this, suppose first that $p_{i}$ is odd or $p_{i}=2$, $2\nmid k$. Then $p_{i}^{\xi(p_{i})-1}(p_{i}-1)\mid k$, and so by Euler's totient theorem, $x^{k}\equiv 1$ (mod $p_{i}^{\xi(p_{i})}$) whenever $p_{i}\nmid x$. Also, $p_{i}^{k}\equiv 0$ (mod $p_{i}^{\xi(p_{i})}$), because $k\ge p_{i}^{\xi(p_{i})-1}(p_{i}-1)\ge \xi(p_{i})$ and so by equation \eqref{main5proofeq1} and the definition of $n$, taking $x_{j}=1$ for all $j\le \tilde{s}$ will give a solution. When $p_{i}=2$ and $2\mid k$, we have that $2^{\xi(2)-2}\mid k$. Every element of $(\bZ/2^{\xi(2)}\bZ)^{\times}$ has order dividing $2^{\xi(2)-2}$ because $\xi(2)\ge 3$, so $x^{k}\equiv 1$ (mod $2^{\xi(2)}$) whenever $2\nmid x$. Because $k>2$, we have $2^{k}\equiv 0$ (mod $2^{\xi(2)}$) and so as before, $x_{j}=1$ for $j\le \tilde{s}$ is a solution.

Recall that for any prime $p$, at least $3k$ of the $a_{1},\dots,a_{\tilde{s}}$ are not divisible by $p$, and so by Lemma \ref{chinonzero2}, $\chi_{p}(\tilde{\ba},n)>0$ for all primes $p$ other than $p_{1},\dots,p_{r}$. The condition on the number of positive and negative $a_{j}$ implies that $a_{1},\dots,a_{\tilde{s}}$ are not all the same sign as well. Hence by Lemma \ref{inhom}, for sufficiently large $s$, there are always primes $x_{1},\dots,x_{\tilde{s}}$ such that
\[a_{1}x_{1}^{k}+\cdots+a_{\tilde{s}}x_{\tilde{s}}^{k}=n.\]
But recalling the definition of $n$, we have primes $x_{1},\dots,x_{s}$ such that
\[a_{1}x_{1}^{k}+\cdots+a_{s}x_{s}^{k}=0.\]
This proves the claim when $k>2$.

When $k=2$, we have $r=2$ and $p_{1}=2$, $p_{2}=3$, $\xi(2)=3$, $\xi(3)=1$. The place where the above argument fails is that $2^{\xi(2)}=8$ and we no longer have $2^{k}\equiv 0$ (mod $2^{\xi(2)})$. Let $m=a_{1}+\cdots+a_{s}$. At least $8$ of the $a_{j}$ are odd, and so by the Cauchy-Davenport lemma, there exists $u_{1}\le 8$ such that, after reordering,
\[3(a_{s-u_{1}+1}+\cdots+a_{s})\equiv -m\: \left(\text{mod $8$}\right).\]
Now let, $x_{j}=2$ for $s-u_{1}+1\le j\le s$. Also choose $u_{2}$ with $u_{2}-u_{1}\le 3$ such that, after reordering,
\[a_{1}+\cdots+a_{s-u_{2}}+a_{s-u_{1}+1}+\cdots+a_{s}\equiv 0\: \left(\text{mod $3$}\right).\]
Let $x_{j}=3$ for $s-u_{2}+1\le j \le s-u_{1}$ and as before, let $\tilde{s}=s-u_{2}$ and
\[n=-(a_{\tilde{s}+1}x_{\tilde{s}+1}^{k}+\cdots+a_{s}x_{s}^{k}).\]
Then $x_{1}=\cdots=x_{\tilde{s}}=1$ is a solution of
\[a_{1}x_{1}^{2}+\cdots+a_{\tilde{s}}x_{\tilde{s}}^{2}\equiv n\: \text{(mod $8\cdot 3$)}.\]
Hence, if $\tilde{a}=(a_{1},\dots,a_{\tilde{s}})$ then $\chi_{2}(\tilde{a},n)>0$, $\chi_{3}(\tilde{\ba},n)>0$. By a similar argument to earlier, we have $\chi_{p}(\tilde{a},n)>0$ for all $p>3$ as well, and $a_{1},\dots,a_{\tilde{s}}$ are not all the same sign, so Lemma \ref{inhom} gives the claim.

Now we just need to bound $d_{j}(s)$. For $d_{2}(s)$, we have that the probability that exactly $v$ of the $a_{j}$ are positive is $\binom{s}{v}2^{-s}$, so summing this over $v\le t$ gives
\[d_{2}(s)=2^{-s}\sum_{v=0}^{t} \binom{s}{v}\ll_{t} s^{t}2^{-s},\]
where the implied constant depends on $t$, but this is not a problem because $t$ is a function of $k$.

Now consider $d_{1}(s)$, $s>t$. If $a_{1},\dots,a_{s}$ are not $(s-t+1)$-wise coprime then neither are $|a_{1}|,\dots,|a_{s}|$. By Lemma \ref{kcoprimedense} with $r=1$, $\bq=\bb=(1,\dots,1)$, $\bu=(1,\dots,1)$ (or just by the main result of \cite{h2013}) we have that the probability of this is $1-A_{s,s-t+1}$, where $A_{s,s-t+1}=A_{s,s-t+1,1}$. Therefore:
\begin{align*}d_{1}(s)&\le 1-A_{s,s-t+1}=1-\prod_{p}\sum_{m=0}^{s-t}\binom{s}{m}(1-p^{-1})^{s-m}p^{-m}\\&=1-\prod_{p}\left(1-\sum_{m=0}^{t-1}\binom{s}{m}(1-p^{-1})^{m}p^{-(s-m)}\right)\\ &\ll \sum_{p}\sum_{m=0}^{t-1}\binom{s}{m}(1-p^{-1})^{m}p^{-(s-m)}\\ &\ll \sum_{p}s^{t-1}p^{-(s-t+1)}\ll s^{t-1}2^{-(s-t+1)}\ll s^{t-1}2^{-s}.\end{align*}
Hence we may take $\lambda(k)=t$.
\end{proof}

\section{Counterexamples and a partial converse}\label{countersection}
In this section we prove some further results about the prime Hasse principle. The first is a proof of Theorem \ref{main6}, a partial converse.
\begin{proof}[Proof of Theorem \ref{main6}]
Suppose that $s\ge 3$ and $\lambda>0$. Suppose for some $\ba$ that $(*)$ does not hold. Then there exists a solution $\bx$ with $x_{j}>|\ba|^{\lambda}$ prime and there exists $p$ such that \eqref{maineq} has no solutions in $\bZ_{p}^{\times}$. We must have $x_{j}=p$ for some $j$, because otherwise $x_{j}\in \bZ_{p}^{\times}$ for all $j$. Hence $p>|\ba|^{\lambda}$. If we suppose $\lambda\ge 1$, then $p>|\ba|$, but this contradicts that $p\mid a_{j}$ for some $j$. Hence we get the first part. For the second part, suppose now that $s\ge 4$, $\lambda>0$. The result certainly holds when $\lambda\ge 1$, so suppose $\lambda<1$. Now if $|\ba|\ge p_{0}(s,k)^{1/\lambda}$, then $p\ge p_{0}(s,k)$ and by Lemma \ref{chinonzero}, $p$ must divide at least $s-2$ of the $a_{j}$. Therefore, for all $A/2\ge p_{0}(s,k)^{1/\lambda}$:
\begin{align}\label{main6proofeq1}
&\#\{A/2<|\ba|\le A: (*)\text{ does not hold}\}\nonumber\\ &\le \sum_{p>(A/2)^{\lambda}}\#\{A/2<|\ba|\le A: p\mid a_{j}\text{ for at least $s-2$ of the $a_{j}$}\}\nonumber\\ &\ll \sum_{p>(A/2)^{\lambda}} \frac{A^{s}}{p^{s-2}}\ll_{\lambda} A^{s-\lambda(s-3)}.\end{align}
To get the result, apply \eqref{main6proofeq1} with $A$ replaced by $A/2^{m}$ for each $m\ge0$ with $A/2^{m+1}\ge p_{0}(k)^{1/\lambda}$ to get
\[\#\{A/2^{m+1}<|\ba|\le A/2^{m}: \text{($\ast$) does not hold}\}\ll_{\lambda} A^{s-\lambda(s-3)}2^{-m(s-\lambda(s-3))}.\]
Now we sum over all $m$, noting that $s-\lambda(s-3)>0$ and that the contribution from the $m$ with $A/2^{m+1}< p_{0}(k)^{1/\lambda}$ is $O(1)$. The result then follows.
\end{proof}
\bigskip
Our final results are some counterexamples to the prime Hasse principle. The Hasse-Minkowski theorem states that the Hasse principle holds for all equations of the form $f(\bx)=0$ with $f$ a quadratic homogeneous integer polynomial (integral quadratic form). A natural question to ask is whether this is still true for the prime Hasse principle \ref{pHp}. The answer turns out to be no, and in the following theorem we construct an infinite set of counterexamples which are diagonal in $3$ variables.

At the heart of the construction lie Pythagorean triples. A Pythagorean triple is $(a,b,c)$ where $a,b,c$ are positive integers and $a^{2}+b^{2}=c^{2}$. It is called primitive if $(a;b;c)=1$ and in this case exactly one of $a$ or $b$ is even, so we may suppose that $b$ is even without loss of generality. The following lemma is classical.
\begin{lemma}\label{pythag_param}
Let $(a,b,c)$ be a primitive Pythagorean triple with $b$ even. Then there exist coprime integers $m>n>0$ with $m\not\equiv n$ (mod $2$) such that
\begin{align*}a&=m^{2}-n^{2}\\b&=2mn\\
c&=m^{2}+n^{2}.\end{align*}
\end{lemma}
\begin{lemma}\label{counter_lemma}
Let $a,b,c$ be fixed positive integers with $b$ even and consider the equation:
\begin{equation}\label{squarecoef}
a^{2}x^{2}+b^{2}y^{2}-c^{2}z^{2}=0.
\end{equation}
Let $S$ be a finite set positive integers and define $\mathcal{P}(S)$ to be the set of all $n\in \bZ$ of the form $n=kl$ where $k\in S$ and $l$ is either $1$ or a prime.

Let $h=\max S$. Suppose that \eqref{squarecoef} has a solution with $x,y\in \mathcal{P}(S)$, $z\in \bZ$ and $(ax;by;cz)=1$. Then $x\le h^{2}a(a+b)$, $y\le h^{2}b(a+b/2)$ and 
\[z\le \frac{h^{2}}{c^{2}}\left(a^{2}+ab+\frac{b^{2}}{2}\right).\]
\begin{remark}
We are only really interested in the case $S=\{1\}$, but we have stated the more general result as it follows from the same method.
\end{remark}
\begin{proof}
Suppose $x,y\in \mathcal{P}(S)$, $z\in \bZ$ is a solution. Then if $u=ax$, $v=by$ and $w=cz$, we have $u^{2}+v^{2}=w^{2}$. Without loss of generality $w>0$ and by assumption $(u;v;w)=1$, so $(u,v,w)$ is a primitive Pythagorean triple. Let $m,n$ be integers given by Lemma \ref{pythag_param}. Since $x\in \mathcal{P}(S)$ we may let $x=kl$ for $k\in S$ an integer and $l$ either $1$ or a prime. It follows that
\[u=akl=m^{2}-n^{2}=(m-n)(m+n).\]
Since $l$ is either $1$ or a prime and $m\not\equiv n$ (mod $2$), we have that at least one of $(m\pm n)$ is coprime to $l$. This means at least one of $(m-n)\mid ak$ or $(m+n)\mid ak$ holds. This implies $0<m-n\le ak\le ah$. Similarly, using $v=2mn$ and $n<m$ we get that $n\le bh/2$. Hence $m\le h(a+b/2)$ and $m+n\le h(a+b)$. So
\[x=\frac{(m-n)(m+n)}{a}\le h^{2}(a+b),\]
\[y=\frac{2mn}{b}\le h^{2}\left(a+\frac{b}{2}\right)\]
and
\[z=\frac{m^{2}+n^{2}}{c}\le \frac{h^{2}}{c}\left(a^{2}+ab+\frac{b^{2}}{2}\right).\]
\end{proof}
\end{lemma}
When $c$ is sufficiently large compared to $a,b,h$ we see that the upper bound for $z$ is less than $1$ and so there are no solutions of the type specified. Our goal in constructing counterexamples will be to find a family of equations of this form which are locally soluble and in which $c$ can become arbitrarily large.
\begin{thm}\label{counterpythag}
Let $(a,b,c)$ be a primitive Pythagorean triple (with $b$ even) and let $r$ be an integer with the following properties:
\begin{enumerate}
\item $r>1$ and $r$ is not prime \item For all primes $p\mid r$ we have $p\equiv 1$ (mod $4$)\item$(r;a)=(r;b)=1$.
\end{enumerate}
Then the equation
\begin{equation}\label{counter1}
a^{2}x^{2}+b^{2}y^{2}-c^{2}r^{2}z^{2}=0
\end{equation}
(in variables $x,y,z$) has solutions in $\bR^{+}$ and $\bZ_{p}^{\times}$ for all primes $p$. Let $\mathcal{P}=\mathcal{P}(\{1\})=\{\text{prime numbers}\}\cup \{1\}$.
Then there exists an effective constant $C(a,b,c)$ such that if $x,y\in \mathcal{P}$ and $z\in \bZ$ is a solution of \eqref{counter1}, then
\[z\le \frac{C(a,b,c)}{r}.\]
Hence, if $r>C(a,b,c)$ then \eqref{counter1} has no solutions in the primes and so \eqref{counter1} does not satisfy the prime Hasse principle \ref{pHp}.
\begin{proof}
It is clear that there is a solution in $\bR^{+}$. We show that there is a solution in every $\bZ_{p}^{\times}$. Suppose first that $p\mid r$. Then $p\equiv 1$ (mod $4$), so $-1$ is a quadratic residue (mod $p$) and $p$ is odd. We also have $a,b\not\equiv 0$ (mod $p$). Let $x\in \bZ$ be such that $x^{2}+1\equiv 0$ (mod $p$). Then
\[a^{2}(bx)^{2}+b^{2}a^{2}\equiv c^{2}r^{2}\equiv 0\text{ (mod $p$)},\]
so there is a solution in $(\bZ/p\bZ)^{\times}$ and as $p$ is odd (and does not divide $a$ or $b$), this can be lifted to a solution in $\bZ_{p}^{\times}$ by Lemma \ref{sollift}.

Now suppose $p\nmid r$. Then $(r,r,1)$ is a solution in $\bZ_{p}^{\times}$.

It just remains to show there are no solutions $x,y\in \mathcal{P}$, $z\in\bZ$ when $r$ is large. Lemma \ref{counter_lemma} with $S=\{1\}$ takes care of the case $(ax;by;crz)=1$, so we may suppose that $(ax;by;crz)>1$. Let $p$ be a prime dividing $(ax;by;crz)$. If $p$ divides $x,y$ and $z$, then in fact $x=y=p$, $z\ge p$ and $(a^{2}+b^{2})p^{2}\ge c^{2}r^{2}p^{2}$. But also $a^{2}+b^{2}=c^{2}$ which is impossible as $r>1$. Hence $p$ divides one of $a,b,cr$. The numbers $a,b,c$ are coprime by assumption, and in fact they must be pairwise coprime because of the equation $a^{2}+b^{2}=c^{2}$. We also have $(a;r)=(b;r)=1$, so $a,b,cr$ are pairwise coprime, and $p$ divides exactly one of them. We therefore have $3$ cases.

If $p\mid a$ then $p\mid y,z$ and so $y=p\le a$. If $x\mid b$ then $x\le b$, so $z\le \sqrt{2}ab/(cr)$, so suppose $x\nmid b$. Hence $(ax/p;b)=1$, so $(ax/p,b,crz/p)$ is a primitive Pythagorean triple (with $b$ even) and so there are coprime $m>n>0$ such that $ax/p=m^{2}-n^{2}$, $b=2mn$ and $crz/p=m^{2}+n^{2}$. We then have $m\le b/2$ and so $x\le ax/p<(b/2)^{2}$. Putting this into $a^{2}x^{2}+b^{2}y^{2}=c^{2}r^{2}z^{2}$ and rearranging gives
\[z\le \frac{1}{cr}\left(\frac{a^{2}b^{4}}{16}+a^{2}b^{2}\right)^{1/2}.\]

If $p\mid b$ then $p\mid x, z$, so $x=p\le b$. By a similar argument to the above, we may assume $y\nmid a$ and there exist $m>n>0$ such that $a=m^{2}-n^{2}$, $by/p=2mn$ and $crz/p=m^{2}+n^{2}$. Then $n<m<m+n\le m^{2}-n^{2}=a$, so $y\le by/p=2mn\le 2a^{2}$ and
\[z\le \frac{1}{cr}\left(a^{2}b^{2}+4a^{4}b^{2}\right)^{1/2}.\]

Finally, if $p\mid cr$ then $x=y=p$ and $a^{2}+b^{2}=c^{2}r^{2}z^{2}/p^{2}$. But also $a^{2}+b^{2}=c^{2}$ and so $r^{2}z^{2}/p^{2}=1$ which implies $p=rz$. But we assumed $r>1$ is not prime so this is impossible.
\end{proof}
\end{thm}
\begin{remark}
If we lose the assumption that $r$ is not prime then there will still be no prime solutions, (and hence \ref{pHp} still fails), but there will be exactly one solution in $\mathcal{P}$, namely $(r,r,1)$. The proof of this is exactly the same.
\end{remark}

We conclude with some counterexamples for $k\ge3$, which are provided by Fermat's last theorem.
\begin{thm}\label{counterfermat}
Let $k\ge 3$ and
\[Q=\prod_{p<p_{0}(3,k)} p^{\xi(p)},\]
where $p_{0}$ is as in Lemma \ref{chinonzero} and $\xi(p)$ is as defined in \eqref{xidef}.

Let $a,b,c$ be pairwise coprime positive integers such that $Q\mid ab$, and for every prime $p\mid c$, $-1$ is a $k$th power (mod $p$). Then the equation
\[a^{k}x^{k}+b^{k}y^{k}-c^{k}z^{k}=0\]
is a counterexample to the prime Hasse principle.
\begin{proof}
Firstly, it follows from Fermat's last theorem that there are no positive integer solutions and hence no prime solutions. The equation clearly always has solutions in $\bR^{+}$, so it remains to show that there is a solution in $\bZ_{p}^{\times}$ for all $p$. By Lemma \ref{sollift}, it suffices to find a solution in $(\bZ/p^{\xi(p)}\bZ)^{\times}$. Suppose $p<p_{0}(3,k)$. Because $Q\mid ab$ and $(a;b)=1$, we find that $p^{\xi(p)}$ divides one of $a,b$ and the remaining coefficients are invertible (mod $p$). Therefore, if $p\mid a$ then $x=1$, $y=c$, $z=b$ is a solution in $(\bZ/p^{\xi(p)}\bZ)^{\times}$ and if $p\mid b$ then $x=c$, $y=1$, $z=a$ is a solution.

If $p\ge p_{0}(3,k)$ and $p\mid a$ or $p\mid b$ then $p\nmid k$, so $\xi(p)=1$ and by the same argument we get a solution.

If $p\mid c$ then $p\ge p_{0}(3,k)$ and there exists $u$ such that $u^{k}\equiv -1$ (mod $p$). Then $x=ub$, $y=a$ and $z=1$ is a solution in $(\bZ/p\bZ)^{\times}$.

The remaining case is when $p\ge p_{0}(3,k)$ and $p$ does not divide any of $a,b,c$. In this case, Lemma \ref{chinonzero} gives us a solution in $\bZ_{p}^{\times}$.
\end{proof}
\end{thm}

\providecommand{\bysame}{\leavevmode\hbox to3em{\hrulefill}\thinspace}

\end{document}